\documentclass[a4paper,twoside]{article}
\usepackage{makeidx}
\usepackage{amssymb,amsfonts,amsmath,amsbsy,theorem,amscd}
\usepackage{dina42e}
\usepackage{epsfig,psfrag,graphicx,xcolor}
\usepackage{stmaryrd}
\usepackage{mathrsfs}
\usepackage{esint}
\usepackage{mathtools}
\usepackage{doi}
\definecolor{CMUrot}{RGB}{128,18,18}
\definecolor{Gold}{RGB}{238,180,34}
\allowdisplaybreaks[4]

\newcommand{\ol}[1]{\overline{#1}}

\numberwithin{equation}{section}

\newcommand{\R}{\ensuremath{\mathbb{R}}}

\newcommand{\Om}{\ensuremath{\Omega}}

\newcommand{\N}{\ensuremath{\mathbb{N}}}

\newcommand{\Z}{\ensuremath{\mathbb{Z}}}

\newcommand{\SD}{\ensuremath{\mathcal{S}}}

\newcommand{\dist}{\operatorname{dist}}
\newcommand{\sdist}{\operatorname{sdist}}

\newcommand{\sd}{{\rm d}}
\newcommand{\supp}{\operatorname{supp}}
\newcommand{\eps}{\ensuremath{\varepsilon}}

\newcommand{\Div}{\operatorname{div}}

\newcommand{\T}{\ensuremath{\mathbb{T}}}

\newcommand{\no}{\mathbf{n}}

\newcommand{\tn}[1]{\mathbf{#1}}
\def\nn{\mathbf{n}}
\let\e=\varepsilon
\newcommand{\ve}{\mathbf{v}}
\newcommand{\we}{\mathbf{w}}
\newcommand{\ue}{\mathbf{u}}

\newcommand{\btau}{{\boldsymbol{\tau}}}
\newtheorem{thm}{THEOREM}[section]
\newtheorem{cor}[thm]{Corollary}
\newtheorem{lem}[thm]{Lemma}
\newtheorem{defn}[thm]{Definition}
\newtheorem{theorem}[thm]{Theorem}
\newtheorem{prop}[thm]{Proposition}

\newtheorem{claim*}{Claim}

\theorembodyfont{\rmfamily}
\newtheorem{rem}[thm]{Remark}

\newenvironment{proof*}[1]{{\bf Proof
#1:}}{\hspace*{\fill}\rule{1.2ex}{1.2ex}\\ }
\newenvironment{proof}{{\bf
Proof:\,}}{\hspace*{\fill}\rule{1.2ex}{1.2ex}\\ }

\newcommand{\p}{\partial}
\newcommand{\G}{\Gamma}
\def\({\left(}
\def\){\right)}

\newcommand{\Lgrad}{L^{\nabla}}
\newcommand{\Ldelta}{L^{\Delta}}
\newcommand{\Lt}{L^{t}}

\newcommand{\tc}{\hat{c}}
\newcommand{\tv}{\hat{\ve}}
\newcommand{\tp}{\hat{p}}
\newcommand{\tl}{\hat{\lambda}}
\newcommand{\tr}{\hat{r}}

\newcommand{\order}{N}
\newcommand{\zg}{\zeta_\Gamma}

\newcommand{\dif}{\ensuremath{\;\mathrm{d}}}

\newcommand{\NSt}[1]{#1}
\begin{document}
\begin{titlepage}
  \title{Sharp Interface Limit for a Mass-Conserving Navier-Stokes/Allen-Cahn System with Different Viscosities }
  \author{Helmut Abels and Hanifah Mumtaz}
\end{titlepage}

\maketitle
\abstract{
\noindent We perform a rigorous examination of the sharp interface limit of a coupled Navier-Stokes and mass-conserving Allen-Cahn system in a two-dimensional, bounded, and smooth domain as the parameter $\varepsilon > 0$, representing the thickness of the diffuse interface, tends to zero. We prove the convergence of solutions from the mass-conserving Navier-Stokes/Allen-Cahn system to those of its sharp interface limit. In this limit, the interface evolves according to mass-conserving mean curvature flow with a convection term and is coupled to a two-phase Navier-Stokes system with surface tension. Our approach entails the construction of an approximate solution for the limiting system through the use of matched asymptotic expansions, complemented by a special ansatz for the leading-order term. In order to estimate the error between this approximate solution and the exact solution, we employ a refined spectral estimate for the linearized Allen-Cahn operator near the approximate solution.

\bigskip

{\small\noindent
%% Incompressible viscous fluids: Capillarity (surface tension)
}

\section{Introduction and Main Result}
The study of two-phase flows involving immiscible fluids represents a fundamental problem with broad applications across a range of scientific disciplines such as material sciences, biomedical and environmental engineering. These flows can be modeled using either a sharp interface or a diffuse interface approach. In sharp interface models, the boundary between the bulk domains is treated as a smooth surface of zero width, with distinct equations governing each fluid and boundary conditions at the interface. In contrast, diffuse interface models account for a finite-width transition layer (of thickness $\varepsilon>0$) described by an order parameter, typically related to fluid concentrations or volume fractions. This approach is advantageous for numerical simulations, as it can handle complex interfacial dynamics, including topological changes like pinch-off and reconnection. The sharp interface limit allows for a rigorous relation between these different kinds of models. 

In this paper we investigate the sharp interface limit $\eps\to 0$ of the following mass-conserving Navier-Stokes/Allen-Cahn system with different viscosities:{\renewcommand{\arraystretch}{0.95}  % Tighten rows slightly (<1.0 compacts)
\setlength{\jot}{12pt}  % Reduce vertical sep (default 10-12pt; try 6-8pt)
\begin{alignat}{2}\label{eq:NSAC1}
  \partial_t \ve_\eps +\ve_\eps\cdot \nabla \ve_\eps-\Div(2\nu(c_\eps)D\ve_\eps)  +\nabla p_\eps & = -\eps \Div (\nabla c_\eps \otimes \nabla c_\eps)                                             & \quad & \text{in}\ \Omega\times(0,T_0), \\\label{eq:NSAC2}
  \Div \ve_\eps                                                                                  & = 0                                                                                            & \quad & \text{in}\ \Omega\times(0,T_0), \\\label{eq:NSAC3}
  \partial_t c_\eps +\ve_\eps\cdot \nabla c_\eps -\Delta c_\eps                                  & =  \frac{1}{\eps^2}  \fint_\Omega f'(c_\varepsilon) \mathrm{d} x-\frac{1}{\eps^2} f'(c_\eps) & \quad & \text{in}\ \Omega\times(0,T_0),
\end{alignat}}
with the boundary and initial conditions
\begin{alignat}{2}\label{eq:NSAC4}
  (\ve_\eps,\partial_{\mathbf{n}}c_\eps)|_{\partial\Omega} & = (0,0)                     & \quad & \text{on }\partial\Omega\times (0,T_0), \\
  \label{eq:NSAC5}
  (\ve_\eps,c_\eps) |_{t=0}                                & = (\ve_{0,\eps},c_{0,\eps}) & \quad & \text{in }\Omega.
\end{alignat}
Throughout this contribution $\Omega \subseteq \R^2$ is a bounded domain with smooth boundary.
The velocity and the
pressure of the fluid mixture are denoted by $\ve_\eps,p_\eps$, respectively, $c_\eps$ is the order parameter, which is related to  the
concentration difference of the fluids, the tensor product $\nabla c_\eps \otimes \nabla c_\eps$ models the surface tension, $\nu(c_\eps)$ describes the viscosity depending on $c_\eps$, and $f$ is a suitable smooth double well potential, i.e., $f\in C^\infty(\R)$, $f'(\pm 1)=0$, $f''(\pm 1)>0$ and $f(s)=f(-s)$ for all $c\in \R$ (e.g., $f(c)=\frac{1}{8}(c^2-1)^2$). Moreover, we assume for simplicity that $|f'''(s)|\leq C(|s|+1)$ for all $s\in\R$ and some $C>0$, which holds true for the standard example as well. The mean value of a function is defined as
$$
\fint_\Omega f(x) \dif x \coloneqq \frac{1}{|\Omega|}\int_\Omega f(x) \dif x.
$$ Furthermore, $D\ve_\eps=\frac 12(\nabla\ve_\eps+(\nabla\ve_\eps)^T)$ is the symmetrized gradient of $\ve_\eps$ and we note that the derivatives $\nabla$= $\nabla_x$, $\Div$= $\Div_x$, $\Delta$= $\Delta_x$ are taken with respect to the spatial variable $x$ only. 
We redefine the mass term as
\begin{equation}
  \lambda_\varepsilon(t) \coloneqq \varepsilon^{-1}\fint_\Omega f(c_\varepsilon(\cdot, t)) \dif x.
\end{equation}
Then the equation \eqref{eq:NSAC3} can be written as
$$\partial_t c_\varepsilon + \mathbf{v}_\varepsilon\cdot \nabla c_\varepsilon = \Delta c_\varepsilon + \frac{1}{\varepsilon}\left(\lambda_\varepsilon(t)-\frac{1}{\varepsilon}f'(c_\varepsilon)\right).$$
We can examine the temporal evolution of the solution's mass for equations \eqref{eq:NSAC3}-\eqref{eq:NSAC4}: $$\partial_t \int_\Omega c_\varepsilon\dif x=\int_\Omega \left(\Delta c_\varepsilon +\varepsilon^{-2}\left[\fint_{\Omega}f^{\prime}\left(c_\varepsilon\right)\mathrm{d}x-f^{\prime}\left(c_\varepsilon\right)\right]-\mathbf{v}_\varepsilon\cdot\nabla c_\varepsilon\right)\dif{x}=0$$ by the divergence theorem and Green's identity. Hence we find that the integral $\int_\Omega c_\varepsilon \dif x$ remains unchanged over time, demonstrating mass conservation in the system.

The non-mass-conserving case (i.e. without $\fint_\Omega f'(c_\varepsilon) \dif{x}$) of this model was originally proposed by Liu and Shen in \cite{LiuShenModelH} as an alternative approximation to the classical sharp interface model for two-phase flows of viscous, incompressible, Newtonian fluids. Subsequently, Jiang et al.~\cite{TwoPhaseVariableDensityJiangEtAl} derived a more general version of the model that allows for fluids with different densities and phase transitions. In addition, the existence of weak solutions and the long-term behavior of the model were investigated. The long-term dynamics of the solutions has been further studied by Gal and Grasselli~\cite{GalGrasselliDCDS}. Recent analytical results on a mass-conserving Navier-Stokes/Allen-Cahn system, along with additional references, can be found in Giorgini et al.~\cite{GiorginiGrasselliWu}.

We will show that solutions of  \eqref{eq:NSAC1}-\eqref{eq:NSAC5} converge to solutions of the following limit system as $\eps\to 0$ (under suitable conditions): For $t\in (0,T_0)$, we have
{\renewcommand{\arraystretch}{0.95}  % Tighten rows slightly (<1.0 compacts)
    \setlength{\jot}{9pt}  % Reduce vertical sep (default 10-12pt; try 6-8pt)
    \begin{alignat}{2}
      \label{eq:Limit1}
      \p_t \ve_0^\pm+\ve_0^\pm\cdot\nabla\ve_0^\pm-\nu^\pm\Delta \ve_0^\pm  +\nabla p^\pm_0 & = 0                                                         & \qquad & \text{in }\Omega^\pm (t),               \\\label{eq:Limit2}
      \Div \ve_0^\pm                                                                        & = 0                                                         & \qquad & \text{in }\Omega^\pm (t) ,              \\\label{eq:Limit3}
      \llbracket 2\nu^\pm D\ve_0^\pm -p_0^\pm \tn{I}\rrbracket\no_{\Gamma_t}                & =- \sigma H_{\Gamma_t}\no_{\Gamma_t}                        &        & \text{on }\Gamma_t,                     \\ \label{eq:Limit4}
      \llbracket\ve_0^\pm \rrbracket                                                        & =0                                                          &        & \text{on }\Gamma_t,                     \\
      \label{eq:Limit5}
      V_{\Gamma_t} -\no_{\Gamma_t}\cdot \ve_0^\pm                                           & = H_{\Gamma_t}-\fint_{\Gamma_{t}} H_{\Gamma_{t}} \dif H^{1} &        & \text{on }\Gamma_t,                     \\
      \ve_0^-|_{\partial\Omega}                                                             & = 0                                                         &        & \text{on }\partial\Omega\times (0,T_0), \\
      \label{eq:Limit6}
      (\ve_0^\pm, \Gamma_t)|_{t=0}                                                          & = (\ve_{0,0}^\pm, \Gamma_0).
    \end{alignat}}
    Here $\nu^\pm=\nu(\pm 1)$ and
$\Omega$ is the disjoint union of two open smooth domains  $\Omega^+(t)$ and $\Omega^-(t)$. Moreover, $\Gamma_t$ smoothly evolving for every $t\in[0,T_0]$ with $\Gamma_t=\partial\Omega^+(t)$,  $\no_{\Gamma_t}$ is the interior normal of $\Gamma_t$ with respect to  $\Omega^+(t)$ and 
\begin{equation*}
  \llbracket u\rrbracket(p,t)= \lim_{h\to 0+} \left[u(p+\no_{\Gamma_t}(p)h)- u(p-\no_{\Gamma_t}(p)h)\right]
\end{equation*}
is the jump of a function $u\colon \Omega\times [0,T_0]\to \R^2$ at $\Gamma_t$ in direction of $\no_{\Gamma_t}$. As usual $H_{\Gamma_t}$ and $V_{\Gamma_t}$ denote the curvature and the normal velocity of $\Gamma_t$, respectively,  both with respect to $\no_{\Gamma_t}$. Finally, $\sigma= \int_{\R}\theta_0'(\rho)^2\sd \rho$, where $\theta_0$ is the optimal profile that is the unique solution of
\begin{alignat}{1}\label{eq:OptProfile1}
  -\theta_0''(\rho)+f'(\theta_0(\rho))   & =0\qquad \text{for all }\rho\in\R, \\\label{eq:OptProfile2}
  \lim_{\rho\to\pm \infty}\theta_0(\rho) & =\pm 1,\qquad \theta_0(0)=0.
\end{alignat}
This will play a central role in the leading order asymptotics as $\eps\to 0$. We note that $\sigma= \int_{-1}^1 \sqrt{2f(s)}\dif{s}$.

The strong well-posedness locally in time of the limit system \eqref{eq:Limit1}-\eqref{eq:Limit6} in the case without the ``volume-preservation term'' $\fint_{\Gamma_{t}} H_{\Gamma_{t}} \dif H^{1}$  was studied rigorously in \cite{AbelsMoserNSAC} by the first author and Moser. The result and its proof can be modified to case with volume preservation since the additional term is lower order. Without volume-preservation Liu et al. \cite{LiuSatoTonegawa2} provided a proof of the existence of weak solutions for a non-Newtonian variant of this system. In the paper \cite{GiorginiGrasselliWu}, Giorgini, Grasselli and Wu proved the well-posedness of the mass-conserving Navier-Stokes/Allen-Cahn system but with the Flory-Huggins potential. In another key contribution \cite{StokesAllenCahn}, Liu and the first author considered a quasi-stationary Stokes flow, assuming equal viscosities ($\nu^+ = \nu^-$).

Allen’s work in \cite{Allen1977} and \cite{Allen19791085} derived the Allen-Cahn equation by minimizing energy and provided physical arguments for deriving the mean curvature flow.
\iffalse In their seminal work, Gurtin et al. \cite{GurtinTwoPhase,}, Hohenberg and Halperin \cite{HohenbergHalperin} derived model H, which couples Cahn-Hilliard to Navier-Stokes to describe two-phase flow in viscous incompressible fluids. \cite{LeeLowengrub1} introduced the Hele-Shaw-Cahn-Hilliard system as a Boussinesq approximation of model H. Existence results for solutions to model H were demonstrated by Boyer et al. in \cite{ModelH, BoyerModelH}, while Fei et al. \cite{Fei, WangZhangHSCH} proved the existence of solutions for the Hele-Shaw-Cahn-Hilliard system. The existence of solutions for sharp interface limit systems derived from phase field models like model H was discussed in detail in \cite{NSMS, StrongNSMS}. Chen’s work on varifold solutions for sharp interface limits was instrumental in proving convergence in a weak sense over large times \cite{ChenSharpInterfaceLimit}.\fi
In another seminal result, De Mottoni and Schatzman \cite{DeMottoniSchatzman} showed convergence of solutions to the Allen-Cahn equation to mean curvature flow using well-prepared initial data and rigorous matched asymptotic expansions. 
Alternative approaches to sharp interface limits include viscosity solutions, as developed by Evans et al.\ and Kohn et al.\ in \cite{ESS, KKR}. Varifold solutions were also utilized in \cite{Ilmanen, MizunoTonegawa, Kagaya} to study these limits in various contexts related to phase field models. Laux and Simon presented conditional results on convergence using relative entropy methods within a $BV$-setting for Allen-Cahn equation locally in time for two- and more components in \cite{LauxSimon}. A first sharp interface limit result using the relative entropy method was given by Fischer, Laux and Simon in \cite{FischerLauxSimon}. This method was extended by Hensel and Liu in \cite{HenselLiuModelH}  to show the sharp interface limit for the Navier-Stokes/Allen-Cahn system with constant viscosity in two and three dimensions. Moreover, the relative entropy method was also extended to the mass-conserving Allen-Cahn equation by Kroemer and Laux \cite{KroemerLaux} to show convergence to volume-preserving mean curvature flow.
The sharp interface limit of a non-mass-conserving Navier-Stokes/Allen-Cahn system was studied by the first author and Fei in \cite{AbeFei23}, laying the foundation for the work in this paper. Another significant contribution is by Chen, Hilhorst, and Logak \cite{ChenHilhorstLogak}, in which the authors showed convergence to the case of the mass-conserving Allen-Cahn equation, focusing on the formal asymptotics of the leading-order terms and deriving the mass-conserving mean curvature flow. In \cite{AbelsConvectiveAC}, the sharp interface limit was derived by the first author using the method of formally matched asymptotic expansion for convective Allen-Cahn systems for leading orders.
These works collectively advances our understanding of phase field models, sharp interface limits, and their applications across various fluid dynamics systems.

Our main result on the sharp interface limit of the mass-preserving Navier-Stokes/Allen-Cahn system is the following:
\begin{thm}\label{thm:main}
  Let $N\geq 3$, $N\in\N$ and let us consider a smooth solution $(\ve_0^\pm,\Gamma)$ of the limit system \eqref{eq:Limit1}-\eqref{eq:Limit6} for some $T_0\in (0,\infty)$. Then, there are  $c_\mathrm{A}\colon \Om\times [0,T_0]\to \R$, $\ve_\mathrm{A}\colon \Om\times [0,T_0]\to \R^2$ (depending on $\eps$) such that the following holds true: If  $(\ve_\eps,c_\eps)$ are strong solutions of the diffuse system \eqref{eq:NSAC1}-\eqref{eq:NSAC5} with initial values $c_{0,\eps}\colon \Omega\to [-1,1]$, $\ve_{0,\eps}\colon \Omega\to \R^2$, $0<\eps\leq 1$ satisfying
          \begin{equation}\label{initial assumption}
            \|c_{0,\eps}-c_{\mathrm{A}}|_{t=0}\|_{L^2(\Omega)}+ \varepsilon^2\|\nabla(c_{0,\eps}-c_{\mathrm{A}}|_{t=0})\|_{L^2(\Omega)}+ \|\ve_{0,\eps}-\ve_{\mathrm{A}}|_{t=0}\|_{L^2(\Omega)}\leq C\eps^{\order+\frac12}
          \end{equation}
          for all $\eps\in (0,1]$ and some $C>0$, then there are some $\eps_0 \in (0,1]$, $T_1\in (0,T_0]$, $R>0$ such that 
          \begin{subequations}\label{assumptions'}
            \begin{align}
              \sup_{0\leq t\leq T_1} \|c_\eps(t) -c_\mathrm{A}(t)\|_{L^2(\Omega)}+\|\nabla (c_\eps -c_\mathrm{A})\|_{L^2(\Omega\times (0,T_1)\setminus\Gamma(\delta))}                          & \leq R\eps^{\order+\frac12}\label{eq:Error1}, \\
              \|\nabla_\btau(c_\eps -c_\mathrm{A})\|_{L^2(\Omega\times(0,T_1)\cap \Gamma(2\delta))}+ \eps \|\partial_\no(c_\eps -c_\mathrm{A})\|_{L^2(\Omega\times(0,T_1)\cap \Gamma(2\delta))} & \leq R\eps^{\order+\frac12}\label{eq:Error2}, \\
              \|\nabla(c_\eps -c_\mathrm{A})\|_{L^\infty(0,T_1;L^2(\Omega))}+\|\nabla^2(c_\eps -c_\mathrm{A})\|_{L^2(\Omega\times(0,T_1))}                                                      & \leq R\eps^{\order-\frac32}\label{eq:Error3}
            \end{align}
          \end{subequations}
          and
          \begin{equation}
            \label{eq:convVelocityb}
            \|\ve_\eps -\ve_\mathrm{A}\|_{L^\infty(0,T_0;L^2(\Om))}+ \|\ve_\eps -\ve_\mathrm{A}\|_{L^2(0,T_1;H^1(\Om))} \leq C(R)\eps^{N+\frac12} 
          \end{equation}
          hold true for all $\eps \in (0,\eps_0]$ and some $C(R)>0$. {Here $\Gamma(\delta), \Gamma(2\delta)$ are as in Section~\ref{subsec:Coordinates} below for some $\delta>0$ depending only on $\Gamma$.} Moreover,
          \begin{equation*}
            \lim_{\eps\to 0} c_\mathrm{A}= \pm 1 \qquad \text{uniformly on compact subsets of } \Omega^\pm= \bigcup_{t\in [0,T_1]} \Omega^\pm (t) \times \{t\}
          \end{equation*}
          and
          \begin{equation*}
            \ve_\mathrm{A}=\ve_0^\pm + O(\eps) \qquad \text{in }L^\infty(\Om\times (0,T_1))\text{ as }\eps\to 0.
          \end{equation*}
\end{thm}

\begin{rem}[Energy Estimates]
  Here  $c_{\mathrm{A},0}= c_\mathrm{A}|_{t=0}$ and $\ve_{\mathrm{A},0}= \ve_\mathrm{A}|_{t=0}$, where the construction of $(c_\mathrm{A}, \ve_\mathrm{A})$ is discussed in Section~\ref{sec:ApproxSolutions} below. In particular we will have $c_{\mathrm{A},0}= c_{\mathrm{A},0}^0 + O(\eps^2)$ with
  \begin{equation*}
    \begin{split}
      c_{\mathrm{A},0}^0(x) & =\zeta(d_{\Gamma_0}(x))\theta_0\left(\tfrac{d_{\Gamma_0}(x)}\eps\right)+(1-\zeta(d_{\Gamma_0}(x)))\left(
      \chi_{\Omega^+(0)}(x)-\chi_{\Omega^-(0)}(x)\right) \quad \text{for all }x\in \Om,
    \end{split}
  \end{equation*}
  where $d_{\Gamma_0}=d_\Gamma|_{t=0}$ is the signed distance function to $\Gamma_0$ and $\zeta\in C^\infty(\R)$ such that
  \begin{equation}\label{eq:1.34}
    \zeta(z)=1~\text{if}~|z|\leq\delta; ~\zeta(z)=0~\text{if}~|z|\geq 2\delta;~ 0\leq  -z\zeta'(z) \leq 4~\text{if}~ \delta\leq |z|\leq 2\delta.
  \end{equation}
  Moreover, we denote
  \begin{equation*}
    \zg(x,t)= \zeta(d_\Gamma(x,t))\qquad \text{for all }(x,t)\in \Gamma(2\delta)
  \end{equation*}
  and extend it by zero to $\ol\Omega \times [0,T]$.
  We remark that every sufficiently smooth solution of \eqref{eq:NSAC1}-\eqref{eq:NSAC5} satisfies the energy identity
  \begin{equation*}
    \frac{\dif}{\dif t}\left(\int_\Omega \tfrac12{|\ve_\eps|^2} \,\dif x\,
    +\int_\Omega (\tfrac{\eps}2|\nabla c_\eps|^2 + \tfrac1\eps f(c_\eps))\dif{x}\,   \right) = - \int_\Omega(|D\ve_\eps|^2 + \tfrac1\eps|\mu_\eps|^2)\,\dif x\,
  \end{equation*}
  for all $t\in (0,T_0)$, where $\mu_\eps = -\eps \Delta c_\eps + \tfrac1{\eps} f'(c_\eps)- \tfrac1{\eps}\fint_\Omega f'(c_\eps)\,dx $. In particular,
  \begin{align}\nonumber
     & \sup_{t\in [0,T_0]}\int_\Omega \tfrac12{|\ve_\eps(t)|^2} \,\dif x\,
    +\int_\Omega (\tfrac{\eps}2|\nabla c_\eps(t)|^2 + \tfrac1\eps f(c_\eps(t)))\dif{x}\,                                     \\\label{eq:EnergyEstim}
     & \qquad\qquad  + \int_0^{T_0}\int_\Omega (|D\ve_\eps|^2 + \tfrac1\eps|\mu_\eps|^2)\dif{x}\, \, \dif t \leq E_{0,\eps},
  \end{align}
  where
  \begin{equation*}
    E_{0,\eps} \coloneqq\int_\Omega \tfrac12{|\ve_{0,\eps}|^2} \,\dif x\,
    +\int_\Omega (\tfrac{\eps}2|\nabla c_{0,\eps}|^2 + \tfrac1\eps f(c_{0,\eps}))\dif{x}\,.
  \end{equation*}
  Hence the left-hand side in \eqref{eq:EnergyEstim} is uniformly bounded in $\eps \in (0,1)$ if $\sup_{\eps\in (0,1)} E_{0,\eps}<\infty$.
\end{rem}

Suitable approximate solutions to \eqref{eq:NSAC1}-\eqref{eq:NSAC4} in the non-conserved case were already constructed in \cite{AbeFei23} with the method of matched asymptotic expansions. The remaining task is to modify the Allen-Cahn solution to the mass-conserving case with the assistance of \cite{ChenHilhorstLogak} (see Theorem~\ref{thm:approx}). The main novelty compared to \cite{AbeFei23} is the occurrence of a first order term ($O(\eps)$) in the formal asymptotics, which vanishes in \cite[Remark A.6]{AbeFei23} and makes the proof of the central result for the approximate solution (cf. Theorem~\ref{thm:approx}) below more involved. Moreover, the total mass $\int_\Omega u(x,t)\dif x$ for the error $u=c_\eps-c_A$ and $\frac1{\eps^2}\fint_\Omega f'(c_\eps)\dif x$ have to be treated carefully.

This paper is structured as follows: In Section \ref{sec:Prelim}, we will present several preliminary results regarding local coordinates in the vicinity of the interface $\Gamma_t$, the definition of stretched variables, parabolic equations on evolving hypersurfaces, and a certain spectral estimate for the Allen-Cahn operator, which is uniform in the small parameter $\varepsilon$. In Section \ref{sec:ApproxSolutions}, the approximate solution is constructed using the method of matched asymptotic expansions and a novel ansatz in the critical order. Finally, the main result is proved in Section \ref{sec:main result}. The details on the matched asymptotic expansions are given in the appendix \ref{Appendix}.

\section{Preliminaries}\label{sec:Prelim}
In this section, we establish preliminaries as in \cite{AbeFei23}.
\subsection{Coordinates}\label{subsec:Coordinates}
The family $(\Gamma_t)_{t\in[0,T_0]}$ is parametrized through smooth diffeomorphisms $X_0\colon \mathbb{T}^1\times [0,T_0]\to \Om$, where $\partial_s X_0(s,t)\neq 0$ holds for all points $(s,t)\in\T^1\times [0,T_0]$ and $\T^1= \R /\Z$. We can choose $X_0$ such that $X_0(.,t)$ parametrizes $\Gamma_t$ proportional to arc-length, i.e., $|\partial_s X_0(s,t)|=\mathcal{H}^1(\Gamma_t)$ for all $(s,t)\in\T^1\times [0,T]$.
We define the normalized tangent and normal vectors by
\begin{equation*}
  \btau(s,t)= \frac{\partial_{s} X_0(s,t)}{|\partial_{s} X_0(s,t)|},\quad \text{and}\quad \no(s,t)= \begin{pmatrix} 0 & -1\\ 1 & 0 \end{pmatrix} \btau(s,t).
\end{equation*}
These vectors are defined on $\Gamma_t$ at position $X_0(s,t)$, with $\no(s,t)$ serving as the exterior normal relative to $\Omega^-(t)$. Additionally, we establish for $s\in\T^1$, $t\in [0,T_0]$: $$\begin{aligned}
    \no_{\Gamma_t}(x)\coloneqq\no (s,t)\text{ for } x=X_0(s,t)\in \Gamma_t,\quad
    V(s,t)\coloneqq V_{\Gamma_t}(X_0(s,t)),\quad H(s,t)\coloneqq H_{\Gamma_t}(X_0(s,t)).
  \end{aligned}$$
Here, $V_{\Gamma_t}$ and $H_{\Gamma_t}$ denote the normal velocity and mean curvature of $\Gamma_t$ with respect to $\no_{\Gamma_t}$. For tubular neighborhoods of $\Gamma_t$, given a sufficiently small $\delta>0$, the orthogonal projection $P_{\Gamma_t}(x)$ is well-defined and smooth for all points in \begin{equation*}
  \Gamma_t(2\delta) =\{y\in \Omega: \dist(y,\Gamma_t)<2\delta\}.
\end{equation*} The parameter $\delta$ is chosen small enough to ensure $\dist(\partial\Omega,\Gamma_t)>2\delta$ for all $t\in [0,T_0]$.
For points in $\Gamma_t(2\delta)$, we can express their unique decomposition as $x=P_{\Gamma_t}(x)+r\no_{\Gamma_t}(P_{\Gamma_t}(x)),$ where $r=\sdist(\Gamma_t,x)$. The signed distance function is defined as: \begin{equation*}
  d_{\G}(x,t)\coloneqq \sdist (\Gamma_t,x)= \begin{cases} \dist(\Omega^-(t),x)  & \text{if } x\not \in \Omega^-(t), \\
              -\dist(\Omega^+(t),x) & \text{if } x \in \Omega^-(t).\end{cases}
\end{equation*} For any $\delta'\in (0,2\delta]$, we define: \begin{equation*}
  \Gamma(\delta') =\bigcup_{t\in [0,T_0]} \Gamma_t(\delta') \times\{t\}, \qquad \Omega^\pm = \bigcup_{t\in [0,T_0]} \Omega^\pm (t)\times \{t\}
\end{equation*} A key integration formula that we frequently employ is: \begin{equation*}
  \int_{\Gamma_t(\delta')} f(x)\dif{x} = \int_{-\delta'}^{\delta'}\int_{\Gamma_t} f(p+r\no_{\Gamma_t}(p))J(r,p,t)\sd \sigma(p)\sd r,
\end{equation*}
where $J\colon (-2\delta,2\delta)\times \Gamma \to (0,\infty)$ represents a smooth function depending on $\Gamma$.

\paragraph{Important relations:}
Several fundamental relations hold (cf.~\cite[Section~4.1]{ChenHilhorstLogak}): \begin{equation}\label{eq:1.26}
  \nabla d_{\G}(x,t)=\no_{\Gamma_t} (P_{\Gamma_t}(x)),~ \partial_t d_{\G}(x,t)=-V_{\Gamma_t} (P_{\Gamma_t}(x)),~\Delta d_\Gamma(q,t)=-H_{\Gamma_t}(q)
\end{equation} for all $(x,t)\in \Gamma(2\delta)$ and $(q,t)\in\Gamma$. We define
\begin{equation}\label{eq:1.50}
  \partial_{\btau} u(x,t)\coloneqq\btau(S(x,t),t)\nabla_x u(x,t),\quad   \nabla_\btau u(x,t)\coloneqq\partial_{\btau} u(x,t)\btau(S(x,t),t)\quad
\end{equation}
for all $(x,t)\in \Gamma(2\delta)$.
Furthermore, we define pull-backs and push-forwards by
\begin{alignat*}{2}
  (X_0^\ast u)(s,t)      & \coloneqq u(X_0(s,t),t)      & \qquad & \text{for all }s\in\T^1,t\in[0,T_0], \\
  (X_0^{\ast,-1} v)(p,t) & \coloneqq v(X_0^{-1}(p,t),t) & \qquad & \text{for all }(p,t)\in\Gamma
\end{alignat*}
if $u\colon \Gamma\to \R^N$  and $v\colon \G_0\times [0,T_0]\to \R^N$  for some $N\in\N$.

\paragraph{New coordinates:}
We introduce new coordinates in $\Gamma(2\delta)$ through the mapping \begin{equation*}
\begin{aligned}
      X\colon (-2\delta, 2\delta)\times \T^1 \times [0,T_0]&\to \Gamma(2\delta),\\
  (r,s,t)&\mapsto X_0(s,t)+r\no(s,t),
\end{aligned}
\end{equation*} 
where
\begin{equation}\label{eq:1.42}
r=\sdist(\Gamma_t,x), \qquad s= X_{0}^{-1}(P_{\Gamma_t}(x),t)\eqqcolon S(x,t).
\end{equation}

\paragraph{Derivatives in new coordinates:}
For any $\phi$, defined on $\Gamma(2\delta)$, we associate a $\tilde{\phi}$, defined on $(-2\delta,2\delta)\times \Sigma \times [0,T_0]$, via
\begin{equation}\label{eq:1.4}
  \phi(x,t)=\tilde{\phi}(d_{\G}(x,t),S(x,t),t)\quad\text{or equivalently}\quad\phi(X_0(s,t)+r\no(s,t),t)=\tilde{\phi}(r,s,t)
\end{equation}\text{ for $(x,t)\in\Gamma(2\delta)$}. The following derivative relations hold as in \cite{StokesAllenCahn}: \begin{equation}\label{Prelim:1.13}
  \begin{split}
    \partial_t \phi(x,t) & = -V_{\Gamma_t} (P_{\Gamma_t}(x)) \partial_r\tilde{\phi}(r,s,t) + \partial_{t}^\Gamma \tilde{\phi}(r,s,t),                  \\
    \nabla \phi(x,t)     & = \no_{\Gamma_t} (P_{\Gamma_t}(x)) \partial_r\tilde{\phi}(r,s,t) + \nabla^ \Gamma \tilde{\phi}(r,s,t),                      \\
    \Delta \phi(x,t)     & = \partial_r^2\tilde{\phi}(r,s,t) + \Delta d_{\G_t}(x) \partial_r\tilde{\phi}(r,s,t) + \Delta^{\Gamma} \tilde{\phi}(r,s,t),
  \end{split}
\end{equation} where $r$, $s$ are defined through equation \eqref{eq:1.42}. We employ the notation (cf.~\cite[Section~4.1]{ChenHilhorstLogak}): \begin{equation}\label{Prelim:1.12}
  \begin{split}
    \partial_{t}^\Gamma \tilde{\phi}(r,s,t) & = \partial_t \tilde{\phi}(r,s,t) + \p_t S(x,t)\cdot\partial_s \tilde{\phi}(r,s,t) ,                 \\
    \nabla^{\Gamma} \tilde{\phi}(r,s,t)     & = \nabla S(x,t) \p_{s} \tilde{\phi}(r,s,t),                                                         \\
    \Delta^{\Gamma} \tilde{\phi}(r,s,t)     & = (\Delta S)(x,t)\cdot\partial_s \tilde{\phi}(r,s,t)+|\nabla S(x,t)|^2 \p_{s}^2 \tilde{\phi}(r,s,t)
  \end{split}
\end{equation}
In the equation \eqref{Prelim:1.12} the spatial variable $x$ is understood via $x=\no(s,t)r+X_0(s,t)$.
We remark that $\nabla^\G g$ is a function of $(r,s,t)$:
\begin{equation}\label{eq:1.46}
  \nabla^\G g(r,s,t)=(\nabla S)(x,t)\partial_s g(s,t), \qquad \text{where }x=X(r,s,t).
\end{equation}
Thus we  denote 
\begin{equation}\label{eq:1.27}
%  \begin{split}
    (\nabla_\G h)(s,t)\coloneqq (\nabla^\Gamma h)(0,s,t), \
    (\Delta_\G h)(s,t)\coloneqq (\Delta^\G h)(0,s,t),     \
    (D_t h)(s,t)\coloneqq (\p_t^\G h)(0,s,t),
%  \end{split}
\end{equation}
and
\begin{equation}\label{Prelim:1.11}
  \begin{split}
    (\Lgrad h)(r,s,t)\coloneqq (\nabla^\Gamma h)(r,s,t)-(\nabla_\Gamma h)(s,t),  \\
    (\Ldelta h)(r,s,t)\coloneqq (\Delta^\Gamma h)(r,s,t)-(\Delta_\Gamma h)(s,t), \\
    (\Lt h)(r,s,t)\coloneqq (\p_t^\G h)(r,s,t)-(D_t h)(s,t),
  \end{split}
\end{equation}
for all $(s,t)\in \T^1\times [0,T_0]$ and any $h\colon \T^1\times [0,T_0]\to \R$.
We note that the coefficients of the latter remainder operators vanish for $r=0$.

\subsection{Function Spaces}
For an open set $U\subseteq\R^N$, we utilize the following function spaces: $L^p(U)$ represents the classical Lebesgue space with respect to the Lebesgue measure, while $W^m_p(U)$ denotes the $L^p$-Sobolev space with differentiation order $m\in\N_0$. The space $H^s(U)$ indicates the $L^2$-Sobolev space with order $s\in\N$, and $H^s_0(U)$ represents the completion of $C_0^\infty(U)$ within $H^s(U)$. For vector-valued functions, we employ the notation $W^m_p(U;X)$, $L^p(U;X)$, and $H^s(U;X)$. The space $C_{0,\sigma}^\infty(U)$ is the divergence-free $C_0^\infty(U)$ space. The space $L^{p,\infty}(\Gamma_t(2\delta))$ consists of measurable functions $f: \Gamma_t(2\delta)\to \R$ with finite norm \begin{equation*}
  \|f\|_{L^{p,\infty}(\Gamma_t(2\delta))} \coloneqq\left(\int_{\T^1}\operatorname{ess\,sup}_{|r|\leq 2\delta } |f(X_0(s,t)+r\no(s,t))|^p \sd s\right)^{\frac1p}
\end{equation*} A key embedding property states that: \begin{equation}
  H^1(\Gamma_t(2\delta))\hookrightarrow L^{4,\infty}(\Gamma_t(2\delta))
\end{equation} This embedding is derived from the interpolation inequality: \begin{equation*}
  \|f\|_{L^\infty(-2\delta,2\delta)}\leq C\|f\|_{L^2(-2\delta,2\delta)}^{\frac12}\|f\|_{H^1(-2\delta,2\delta)}^{\frac12}\quad \text{for }f\in H^1(-2\delta,2\delta).
\end{equation*}

\subsection{The Stretched Variable and Remainder Terms}
As in \cite{ChenHilhorstLogak} and \cite{AbeFei23}, a stretched coordinate system is introduced through the variable
\begin{equation}\label{eq:StretchedVariable}
  \rho_\varepsilon = \frac{d_\Gamma(x,t)}{\eps}- h_\eps(s,t) \qquad \text{for } (x,t)\in \Gamma(2\delta), \eps \in (0,\eps_0)
\end{equation}
where $s=S(x,t)$ as in \eqref{eq:1.42}. The function $h_\eps: \T^1\times [0,T]\to \R$ is assumed to be sufficiently differentiable with uniform bounds on its $C^k$-norms for large $k\in\N$ and all $\eps\in (0,\eps_0)$. We will often abbreviation $\rho_\varepsilon$ by just $\rho$.
The following result is a consequence of the chain rule and \eqref{Prelim:1.13}, cf.\ \cite[Section~4.2]{ChenHilhorstLogak}:
\begin{lem}\label{lem:ChainRule}
  Let ${\hat{w}}\colon \R\times \Omega\times [0,T_0]\to \R$ be sufficiently smooth and let
  \begin{equation*}
    w(x,t)={\hat{w}}\(\rho(x,t),x,t\) \quad \text{for all }(x,t)\in \Gamma(2\delta).
  \end{equation*}
  Then for each $\eps>0$
  \begin{equation}\label{eq:formula1}
    \begin{split}
      \p_t w(x,t)=   & -\(\tfrac{V_{\Gamma_t} (P_{\Gamma_t}(x))}\eps +\p_t^\Gamma h_\eps(r,s,t)\)\p_\rho {\hat{w}}(\rho,x,t) +\p_t {\hat{w}}(\rho,x,t),         \\
      \nabla w(x,t)= & \(\tfrac{\no_{\Gamma_t} (P_{\Gamma_t}(x))}  \eps -\nabla^\Gamma h_\eps(r,s,t)\)\p_\rho{\hat{w}}(\rho,x,t) +\nabla{\hat{w}}(\rho,x,t),    \\
      \Delta w(x,t)= & (\eps^{-2}+|\nabla^\Gamma h_\eps(r,s,t)|^2) \p^2_\rho{\hat{w}}(\rho,x,t)                                                                 \\
                     & +\(\eps^{-1}\Delta d_\Gamma (x,t) -\Delta^\Gamma h_\eps(r,s,t)\)\p_\rho{\hat{w}}(\rho,x,t)                                               \\
                     & +2\left({\frac{\no_{\Gamma_t}}\eps}-\nabla^\Gamma h_\eps(r,s,t)\right)\cdot\nabla \p_\rho{\hat{w}}(\rho,x,t)+\Delta {\hat{w}}(\rho,x,t),
    \end{split}
  \end{equation}
  where $\rho$ is as in \eqref{eq:StretchedVariable} and $(r,s)$ is understood via \eqref{eq:1.42}.
\end{lem}
We use the following classes to systematically handle exponentially decaying remainder terms as in \cite[Definition~2.3]{AbeFei23}:
\begin{defn}\label{eq:1.15}
  For any $k\in \R$ and $\alpha>0$,  $\mathcal{R}_{k,\alpha}$ denotes the vector space of families of continuous functions $\tr_\eps\colon \R\times \Gamma(2\delta) \to \R$, indexed by $\eps\in (0,1)$, which are continuously differentiable with respect to $\no_{\Gamma_t}$ for all $t\in [0,T_0]$ such that
  \begin{equation}\label{eq:EstimRkalpha}
    |\partial_{\no_{\Gamma_t}}^j \tr_\eps(\rho,x,t)|\leq Ce^{-\alpha |\rho|}\eps^k\qquad \text{for all }\rho\in \R,(x,t)\in\Gamma(2\delta),\, j\in\{0,1\},\, \eps \in (0,1)
  \end{equation}
  for some $C>0$ independent of $\rho\in \R,(x,t)\in\Gamma(2\delta)$, $\eps\in (0,1)$.  $\mathcal{R}_{k,\alpha}^0$ is the subclass of all $(\tr_\eps)_{\eps\in (0,1)}\in \mathcal{R}_{k,\alpha}$ such that
  $\tr_\eps(\rho,x,t)= 0$ for all  $\rho \in\R, x\in\Gamma_t, t\in [0,T_0]$.
\end{defn}
For remainder estimates, we use:
\begin{lem}\label{lem:rescale}
  Let $0<\eps\leq \eps_0$, $h_\eps$ be as in the beginning of this subsection and satisfy
  \begin{equation*}
    M\coloneqq \sup_{0<\eps <\eps_0, (s,t)\in \T^1\times [0,T_\eps]} |h_\eps (s,t)| <\infty
  \end{equation*}
  for some  $T_\eps \in (0,T_0]$, $\eps_0\in (0,1)$,  and   $(\tr_\eps)_{0<\eps<1}\in \mathcal{R}_{k,\alpha}$ for some $\alpha>0$, $k\in\R$ and let $j=1$ if $(\tr_\eps)_{0<\eps<1}\in \mathcal{R}_{k,\alpha}^0$ and $j=0$ else.
  Then there is some $C>0$, independent of $T_\eps,0<\eps\leq \eps_0$, $\eps_0\in (0,1)$  such that
  \begin{equation*}
    r_\eps (x,t)\coloneqq\tr_\eps\left( \rho, x,t\right)\qquad \text{for all }(x,t)\in\Gamma(2\delta)
  \end{equation*}
  with $\rho$ as in \eqref{eq:StretchedVariable} satisfies
  \begin{align}\label{eq:RemEstim1}
    \left\|{\mathfrak{a}(P_{\Gamma_t}(\cdot))r_\eps \varphi} \right\|_{L^1(\G_t(2\delta))} & \leq C(1+M)^j\eps^{1+k+j}\|\varphi\|_{H^1(\Omega)}\|\mathfrak{a}\|_{L^2(\Gamma_t)}, \\
    \label{eq:RemEstim2}
    \left\| {\mathfrak{a}(P_{\Gamma_t}(\cdot))r_\eps} \right\|_{L^2(\G_t( 2\delta))}       & \leq C (1+M)^j\eps^{\frac 12+k+j} \|\mathfrak{a}\|_{L^2(\Gamma_t)}
  \end{align}
  uniformly for all $\varphi\in H^1(\Omega)$, $\mathfrak{a}\in L^2(\Gamma_t)$, $t\in [0,T_\eps]$, and $\eps\in (0,\eps_0]$.
\end{lem}
We refer to \cite[Corollary 2.7]{StokesAllenCahn} for the proof of the last lemma and for the next result to the prior work \cite[Lemma 2.5]{AbeFei23}:

\begin{lem}\label{lem:DivergenceFreeRemainder}
  Let $f\in \SD(\R)$ such that $f'\in \mathcal{R}_{0,\alpha}$ for some $\alpha>0$. Then there is a constant $C>0$ such that for all $t\in [0,T_0]$, $a\in H^1(\T^1)$ and $\boldsymbol{\varphi}\in C_{0,\sigma}^\infty(\Omega)$ we have
  \begin{equation*}
    \left|\int_{\Gamma_t(2\delta)} f'(\rho(x,t))a(S(x,t))\no_{\Gamma_t}\otimes \no_{\Gamma_t} : \nabla \boldsymbol{\varphi}  \dif x\,\right|\leq C \eps^{\frac32} \|a\|_{H^1(\T^1)}\|\boldsymbol{\varphi}\|_{H^1(\Omega)}.
  \end{equation*}
\end{lem}

The following lemma provides the key quantitative estimates on localized diffuse-interface integrals over the diffuse interface region that will be essential in the proof of Theorem \ref{thm: h_N+1/2}, where the approximate mass conservation property is established.

\begin{lem}\label{lem:MeanEstim}
  Assume that $f\in\mathcal{S}(\R)$. Then there is some $C>0$ such that
  \begin{equation*}
\left|    \int_{\Gamma_t(2\delta)} \zeta_\Gamma f'(\rho_\eps(x,t)) a(S(x,t))\dif{x} \right|\leq C\eps^2\|a\|_{L^1(\T^1)}
\end{equation*}
for all $a\in L^1(\T^1)$, $\eps\in (0,1]$ and $t\in [0,T_0]$.   Moreover, if $a,b\in L^1(\T^1)$ such that
  \begin{equation}\label{eq:Cancelation}
   \int_{\T^1} (a(s)\kappa(s,t)-b(s))\dif{s} =0,
  \end{equation}
  then
  \begin{equation*}
\left|    \int_{\Gamma_t(2\delta)} \zeta_\Gamma (f'(\rho_\eps(x,t)) a(S(x,t))+\eps f(\rho_\eps(x,t))b(S(x,t) )\dif{x} \right|\leq C\eps^3\left(\|a\|_{L^1(\T^1)}+\|b\|_{L^1(\T^1)}\right),
\end{equation*}
for some $C>0$ independent of $a,b$, $t\in[0,T_0]$ and $\eps \in (0,1]$.
\end{lem}
\begin{proof}
Let $t\in [0,T]$ be given. For simplicity we assume that $\mathcal{H}^1(\Gamma_t)=1$ for consistency with the assumptions in \cite{ChenHilhorstLogak}. By a simple scaling we can always reduce to that case. 
 For the proof of the first statement in Lemma \ref{lem:MeanEstim} we use that
    \begin{align*}
    &\int_{\Gamma_t(2\delta)} \zeta_\Gamma f'(\rho_\eps(x,t)) a(S(x,t))\dif{x} = \int_{-2\delta}^{2\delta}\int_{\T^1} \zeta (r) f'(\tfrac{r}\eps-h_\eps(s,t)) a(s)J_t(r,s)\dif{s}\dif{r},
    \end{align*}
    where $J_t(r,s)= 1+ rJ_{1,t}(r,s)$ for some uniformly bounded $J_{1,t}$. Moreover,
    using
    $$
    f'(\tfrac{r}\eps-h_\eps(s,t))= \eps \frac{\dif}{\dif r}f(\tfrac{r}\eps-h_\eps(s,t))
    $$
    and repeated integration by parts, we obtain for any $N\in\N$:
  \begin{align*}
    &\left|\int_{-2\delta}^{2\delta}\int_{\T^1} \zeta (r) f'(\tfrac{r}\eps-h_\eps(s,t)) a(s)\dif{s}\dif{r}\right| \\
    &\leq \eps\int_{-2\delta}^{2\delta}\int_{\T^1} |\zeta' (r)| f(\tfrac{r}\eps-h_\eps(s,t)) a(s)|\dif{s}\dif{r}\leq  C_N\eps^N\|a\|_{L^1(\T^1)}  
  \end{align*}
  for some $C_N$ independent of $\eps, t,a$
  since $f$ is rapidly decreasing and $\supp \zeta'\subseteq [-2\delta,-\delta]\cup [\delta,2\delta] $. Finally,
  \begin{align*}
    &\left|\int_{-2\delta}^{2\delta}\int_{\T^1} \zeta (r) rf'(\tfrac{r}\eps-h_\eps(s,t)) a(s)J_1(r,s)\dif{s}\dif{r}\right|\\
    &\quad \leq C\eps\int_{\R} (|\tfrac{r}\eps|+1)|f(\tfrac{r}\eps)|\,\dif{r} \|a\|_{L^1(\T^1)}  = C'\eps^2\|a\|_{L^1(\T^1)}  
  \end{align*}
  In order to prove the second statement in Lemma \ref{lem:MeanEstim}, we use that
  \begin{equation*}
    J_t(r,s)= 1+ \partial_r J_t(0,s)r + r^2 J_{2,t}(r,s)\qquad \text{for all } r\in (-2\delta,2\delta),s\in\T^1
  \end{equation*}
  and some uniformly bounded $J_{2,t}$. Here $\partial_rJ_t(0,s)= \kappa(s,t)$, cf.\ \cite[Lemma 4]{ChenHilhorstLogak}. Then we obtain 
 \begin{align*}
&\int_{\Gamma_t(2\delta)} \zeta_\Gamma \bigl( f'(\rho_\eps(x,t)) a(S(x,t)) + \eps f(\rho_\eps(x,t)) b(S(x,t)) \bigr) \dif{x} \\
&= \int_{-2\delta}^{2\delta} \int_{\mathbb{T}^1} \zeta(r) f'\!\Bigl( \tfrac{r}{\eps} - h_\eps(s,t) \Bigr) a(s) \bigl( 1 + \kappa(s,t)  r \bigr) \dif{s} \dif{r} \\
&\quad + \eps \int_{-2\delta}^{2\delta} \int_{\mathbb{T}^1} \zeta(r) f\!\Bigl( \tfrac{r}{\eps} - h_\eps(s,t) \Bigr) b(s) \dif{s} \dif{r} \\
&\quad + \eps^2 \int_{-2\delta}^{2\delta} \int_{\mathbb{T}^1} \zeta(r) \Bigl[ \tfrac{r^2}{\eps^2} f'\!\Bigl( \tfrac{r}{\eps} - h_\eps(s,t) \Bigr) a(s) J_{2,t}(r,s) + \tfrac{r}{\eps} f\!\Bigl( \tfrac{r}{\eps} - h_\eps(s,t) \Bigr) b(s) J_{1,t}(r,s) \Bigr] \dif{s} \dif{r} \\
&\equiv I_1 + I_2 + I_3.
\end{align*}  
  Using integration by parts as before, we obtain
  \begin{align*}
    I_1= &-\eps\int_{-2\delta}^{2\delta}\int_{\T^1} \zeta (r) f(\tfrac{r}\eps-h_\eps(s,t)) a(s)\kappa(s,t)\dif{s}\dif{r}\\
 &   - \eps\int_{-2\delta}^{2\delta}\int_{\T^1} \zeta' (r) f(\tfrac{r}\eps-h_\eps(s,t)) a(s)(1+ \kappa(s,t)r)\dif{s}\dif{r}\equiv I_{1,1}+I_{1,2},
  \end{align*}
  where  $|I_{1,2}|\leq C_N\eps^N\|a\|_{L^1(\T^1)}$ for any $N\in\N$ as before. By \eqref{eq:Cancelation}, we get
  \begin{align*}
    I_{1,1}+ I_2= -\eps\int_{-2\delta}^{2\delta}\zeta(r) f(\tfrac{r}\eps-h_\eps(s,t))\int_{\T^1} ( a(s)\kappa(s,t)-b(s))\dif{s}\dif{r}=0.
  \end{align*}
  Finally, similarly as before
  \begin{equation*}
    |I_3|\leq C\eps^3 (\|a\|_{L^1(\T^1)}+\|b\|_{L^1(\T^1)})
  \end{equation*}
  since $f\in\SD(\R)$ and
  \begin{equation*}
    \frac{|r^2|}{\eps^2}\leq C\left( \left(\tfrac{r}\eps-h_\eps(s,t)\right)^2 +1 \right)\quad \text{for all } (r,s,t)\in\R\times \T^1\times [0,T], \eps\in (0,\eps_0]. 
  \end{equation*}
  This finishes the proof.
\end{proof}

\subsection{Spectral Estimate}

For this section, we consider $\Omega\subseteq \R^2$ as a bounded smooth domain and $\Gamma_t\subseteq \Om$, $t\in [0,T_0]$, $T_0>0$, as smoothly evolving compact $C^\infty$-hypersurfaces. These surfaces divide $\Om$ into disjoint domains $\Omega^+(t)$ and $\Omega^-(t)$. We assume for all $x\in \Om$: \begin{equation*}
  \begin{aligned}
    c_\mathrm{A}(x)           & = c_{\mathrm{A},\leq 1}(x)+ \eps^2c_{\mathrm{A},1+}(x),                                                                                                                                                   \\
    c_{\mathrm{A}, \leq 1}(x) & = \zeta \circ d_{\Gamma} \theta_{0}(\rho)+\left(1-\zeta \circ d_{\Gamma}\right)\left(\chi_{\Omega^+ (t)}-\chi_{{\Omega^- (t)}}\right)                                                                     \\
                              & +\varepsilon\lambda_0\left[\zeta \circ d_{\Gamma} \theta_{1}(\rho)+\left(1-\zeta \circ d_{\Gamma}\right)\left(\frac{1}{f''(1)}\chi_{{\Omega^+ (t)}}+\frac{1}{f''(-1)}\chi_{{\Omega^- (t)}}\right)\right],
  \end{aligned}
\end{equation*}
where $\zeta\in C^\infty(\R)$ fulfills the assumption in \eqref{eq:1.34}. Here $\theta_1\in C^1(\mathbb{R})\cap L^\infty(\mathbb{R})$ is the unique solution to
\begin{align}\label{eq:theta1}
  - \theta_1''(\rho)+ f''(\theta_0(\rho))\theta_1(\rho) & =1-\frac{2}{\sigma} \theta_0^{\prime}(\rho) \quad \text {for all }\rho\in \mathbb{R}\text{ with }
                                                          \theta_1(0)  =0,\, 
\end{align}
where $\sigma\coloneqq \int_{\mathbb{R}} \theta_0^{\prime 2}(\rho)\dif \rho$ as before.
We note that $\theta_1$ is an even function since $\theta'_0$ and $f''(\theta_0)$ are even. This yields the important relation %Integrating $\theta_0^{\prime \prime} \mathcal{L} \theta_1$ by parts over $\mathbb{R}$, one can verify that
\begin{equation}
  \label{eq:ortho}
\int_\mathbb{R}\theta_0'(\rho)^2 f'''(\theta_0(\rho))\theta_1(\rho)\dif \rho=0,  
\end{equation}
since $f'''(\theta_0)$ is odd.
 We maintain that $\dist(\Gamma_t,\partial\Omega)> 2\delta$ for all $t\in [0,T_0]$. For given continuous functions $(\tilde{h}_\eps)_{0<\eps< 1} \colon \G\to \R$ with $\Gamma\coloneqq \bigcup_{t\in[0,T_0]} \Gamma_t \times {{t}}$, we already defined the stretched variable as $
  \rho = \frac{d_{\G}(x,t)}{\eps}- h_\eps(s,t)$, where $h_\eps(s,t)= \tilde{h}_\eps(X_0(s,t),t)$. Additionally, we require: \begin{equation*}
  \sup_{\eps \in (0,1)}\left(\sup_{(p,t)\in \Gamma} |\tilde{h}_\eps (p,t)|+ \sup_{x\in\Om, t\in[0,T_0]} |c_{\mathrm{A},1+} (x,t)|\right)\leq M
\end{equation*} These results will be applied to $\tilde{h}_\eps (p,t)= h_\eps(X_0^{-1}(p,t),t)$ for some $h_\eps\colon \T^1\times [0,T_0]\to \R$. The following spectral estimate from \cite{PhDMarquardt} plays a central role in proving convergence. The preceding requirements differ from the ones in \cite{AbeFei23} by the approximate solution $c_\mathrm{A}$, which will be further discussed in the next section and in the Appendix \ref{Appendix}.

\begin{thm}\label{thm:Spectral}
  Let $c_\mathrm{A}$ and $M>0$ be as above and $\alpha=f(\pm 1)$. Then there are some $C_L,\eps_0>0$, which are independent of $\tilde{h}_\eps, c_\mathrm{A}$, such that for every $\psi\in H^1(\Om)$, $t\in[0,T_0]$, and $\eps\in (0,\eps_0]$ we have
  \begin{align*}
    &\int_{\Om}\left(|\nabla\psi(x)|^2+ \frac{f''(c_\mathrm{A}(x,t))}{\eps^2}\psi^2(x)\right)\dif{x}\\
    &\quad \geq -C_L\int_\Om \psi^2 \dif{x} + \int_{\Om\setminus \Gamma_t(\delta)}|\nabla \psi|^2\dif{x} +  \int_{\Gamma_t({\delta})} |\nabla_\btau \psi|^2\dif{x}.
  \end{align*}
\end{thm}
\begin{proof}
  The estimate is shown in detail in \cite[Equation (3.75)]{PhDMarquardt}. We note that Assumption 3.1 in this thesis is satisfied due to \eqref{eq:ortho}.
\end{proof}

The following decomposition of $\psi$ from \cite[Corollary~2.12]{AbelsMarquardt1} is important throughout the following analysis.
\begin{cor}
  \label{cor:SpectralDecomp}
 Let the previous assumptions hold true and let
$t\in\left[0,T\right]$, let $\psi\in H^{1}(\Gamma_{t}(\delta))$
and $\Lambda_{\eps}\in\mathbb{R}$ be such that
\begin{equation}
\label{def Lambda energy}
\int_{\Gamma_{t}(\delta)}\eps\left|\nabla\psi(x)\right|^{2}+\eps^{-1}f''\left({c_{A}}(x,t)\right)\psi(x)^{2}\dif{x}\leq\Lambda_{\eps}
\end{equation}
and denote $I_{\eps}^{s,t}\coloneqq\left(-\frac{\delta}{\eps}-h_{{\eps}}(s,t),\frac{\delta}{\eps}-h_{{\eps}}(s,t)\right)$.
Then, for $\eps>0$ small enough, there exist functions $Z\in H^{1}(\mathbb{T}^{1})$,
$\psi^{\mathbf{R}}\in H^{1}(\Gamma_{t}(\delta))$
and smooth $\Psi\colon I_{\eps}^{s,t}\times\mathbb{T}^{1}\to \R$ such that
\begin{equation}
\psi\left(X(r,s,t)\right)=\eps^{-\frac{1}{2}}Z(s)\left(\beta(s)\theta_{0}'(\rho(r,s))+\Psi(\rho(r,s),s)\right)+\psi^{\mathbf{R}}(r,s)\label{decompose u}
\end{equation}
for almost all $\left(r,s\right)\in\left(-\delta,\delta\right)\times\mathbb{T}^{1}$,
where $\rho(r,s)=\frac{r}{\eps}-h_{{\eps}}(s,t)$
and $\beta(s)=\left(\int_{I_{\eps}^{s,t}}\left(\theta_{0}'(\rho)\right)^{2}\sd\rho\right)^{-\frac{1}{2}}$.
Moreover,
\begin{equation}
\Vert \psi^{\mathbf{R}}\Vert _{L^{2}\left(\Gamma_{t}(\delta)\right)}^{2}\leq C\left(\eps\Lambda_{\eps}+\eps^{2}\left\Vert \psi\right\Vert _{L^{2}\left(\Gamma_{t}(\delta)\right)}^{2}\right),\label{f2u estimate-1}
\end{equation}
\begin{equation}\label{f2u estimate-2}
\Vert Z\Vert _{H^{1}\left(\mathbb{T}^{1}\right)}^{2}+\Vert \nabla^{\Gamma}\psi\Vert _{L^{2}(\Gamma_{t}(\delta))}^{2}+\Vert \psi^{\mathbf{R}}\Vert _{H^{1}\left(\Gamma_{t}(\delta)\right)}^{2}\le C\left(\left\Vert \psi\right\Vert _{L^{2}\left(\Gamma_{t}(\delta)\right)}^{2}+\frac{\Lambda_{\eps}}{\eps}\right),
\end{equation}
and
\begin{equation}\label{f1u estimate}
\sup_{s\in\mathbb{T}^{1}}\left(\int_{I_{\eps}^{s,t}}\left(\Psi(\rho,s)^{2}+\Psi_{\rho}(\rho,s)^{2}\right)J\left(\eps(\rho+h_{{\eps}}(s,t)),s\right)\sd\rho\right)\leq C\eps^{2}.%\label{eq:Chenneu4}
\end{equation}
\end{cor}

As in \cite[Section~2.5]{AbeFei23}  we introduce $\eps$-dependent norms for $u\in H^1(\Gamma_t(2\delta))$ 
  \begin{align*}
    \|u\|_{X_\eps} = &\inf\left\{ \|Z\|_{H^1(\Gamma_t)}+ \|v\|_{H^1(\Gamma_t(2\delta))}+\eps^{-1}\|v\|_{L^2(\Gamma_t(2\delta))}:\right.\\
    &\quad \qquad \left. \tilde{u}(\rho,s)= Z(s) \eps^{-\frac12} \theta_0'(\rho) + \tilde{v}(r,s), Z\in H^1(\Gamma_t), v\in H^1(\Gamma_t(2\delta)) \right\}.
  \end{align*}
  We remark that
  \begin{equation*}
    \|u\|_{L^2(\Gamma_t(2\delta))}\leq C\|u\|_{X_\eps}\quad\text{for all }u\in H^1(\Gamma_t(2\delta)), t\in[0,T],\eps \in (0,\eps_0].
  \end{equation*}
  Then, if $\Lambda_\eps$ is as in \eqref{def Lambda energy},} we obtain
  \begin{equation*}
    \|u\|_{X_\eps}^2+\|\nabla^\Gamma u \|_{L^2(\Gamma_t(2\delta))}^2 \leq C\left(\int_{\Gamma_t(2\delta)} \left(|\nabla u|^2 +\frac1{\eps^2} f''(c_A(x,t))u^2\right)\, dx + \|u\|_{L^2(\Gamma_t(\delta))}^2 \right).
  \end{equation*}
  Here we replaced $\delta$ by $2\delta$ for consistency in the following. Since the size of $\delta>0$ does not matter, this can always be done.
We refer to \cite[Section~2.5]{AbeFei23} for more details.

\section{Construction of the Approximate Solutions}\label{sec:ApproxSolutions}

In this section we will construct suitable approximate solutions to the mass-conserved Navier-Stokes/Allen-Cahn system, which is a key ingredient in the proof of our main result. We will follow the strategy and arguments in \cite[Section~3]{AbeFei23}. But the construction and proofs have to be adapted carefully to the case of a mass-conserved Allen-Cahn equation since several new terms occur. The main result of this section is summarized in:
\begin{thm}\label{thm:approx}
  Let $M>0$, $\eps_0 >0$, and  $T_\eps \in (0,T_0]$ for all $\eps\in (0,\eps_0)$ be given. Then for smooth $\ue=\ue_\eps\colon [0,T_\varepsilon)\times\Omega \to \R^2$, depending on $\eps \in (0,\eps_0)$ and satisfying 
  \begin{equation*}
    \|\ue \|_{L^2(0,T_\eps; H^1(\Omega))}\leq M\qquad \text{for all }\eps \in (0,\eps_0),
  \end{equation*}
  there are smooth $c_\mathrm{A},p_\mathrm{A}\colon \Omega\times [0,T_\eps]\to \R$, $\ve_\mathrm{A}\colon \Omega\times [0,T_\eps]\to \R^2$ and $\lambda_\mathrm{A} \colon [0,T_\varepsilon] \to \R$ such that
  \begin{alignat}{2}\label{eq:ApproxS1}
    \partial_t \ve_\mathrm{A}\NSt{+ \ve_\mathrm{A}\cdot \nabla \ve_\mathrm{A}} -\Div (2\nu(c_\mathrm{A})D \ve_\mathrm{A}) +\nabla p_\mathrm{A} & = -\eps \Div (\nabla c_\mathrm{A}\otimes \nabla c_\mathrm{A}) + R_\eps^1+  R_\eps^2,
    \\\label{eq:ApproxS2}
    \Div \ve_\mathrm{A}                                                                                                                        & =\, 
    G_\eps,                                                                                                                                                                                                                                     \\ \label{eq:approxAC}
    \partial_t c_\mathrm{A} + \ve_\mathrm{A}\cdot \nabla c_\mathrm{A}+\eps^{N+\frac12}\zg\ue|_{\Gamma} \cdot \nabla c_\mathrm{A}                  & = \Delta c_\mathrm{A} -\frac1{\eps^2} f'(c_\mathrm{A})+\frac{1}{\varepsilon}\lambda_\mathrm{A} 
    + S_\eps,                                                                                                                                                                                                                                   \\
    (\ve_\mathrm{A},\partial_\mathbf{n}c_\mathrm{A})|_{\partial\Omega}                                                                         & = (0,0),                                                                                       \\\label{eq:ApproxS5}
    \frac{\dif}{\dif {t}}\int_\Omega c_\mathrm{A}(x,t) \dif x                                                                                       & = \bar{S}_\eps,
  \end{alignat}
  where
  \begin{alignat}{2}\nonumber
    \|R_\eps^1\|_{L^2(0,T_\eps; (H^1(\Omega)^2)'))}                                   & \leq {C'}\|\ue|_{\Gamma}\|_{L^2(0,T_\varepsilon;L^2(\Gamma_t))} \eps^{N+\frac12}, \\\nonumber
    \|R_\eps^2\|_{L^2(0,T_\eps; (H^1(\Omega)^2)'))}                                   & \leq C(M) \eps^{N+1},                                                             \\\label{eq:SepsEstim1}
    \|G_\eps\|_{H^{1/2}(0,T_\eps; L^2(\Omega))}+\|S_\eps\|_{L^2(0,T_\eps; (X_\eps)'\cap L^1(\Gamma_t(2\delta))} & \leq C(M)\eps^{N+1},                                                              \\\nonumber
    \|S_\eps\|_{L^2(0,T_\eps; L^2(\Omega))}                                           & \leq C(M)\eps^N,\\\nonumber
    \|\bar{S}_\eps\|_{L^\infty(0,T_\eps)}&\leq C(M)\eps^{N+\frac32} 
  \end{alignat}
  for some $C(M)$, $C'>0$ independent of $\eps, T_\eps$. 
    Moreover, $M\mapsto C(M)$ is increasing.
\end{thm}
We note that the residual terms are derived and estimated similarly as in \cite{AbeFei23}, where the non-mass-conserving case is treated (with Dirichlet instead of Neumann boundary conditions). However, for the approximate solutions modifications as in \cite{ChenHilhorstLogak} are employed here. We construct approximate solutions of the Navier-Stokes/Allen-Cahn system in the following form
\begin{alignat}{1}\label{eq:DefcA}
  c_\mathrm{A} (x,t)            & =\zg c_\mathrm{A}^{\text{in}}(x, t)+\left(1-\zg\right)\left(c_\mathrm{A}^{+} \chi_{+}+c_\mathrm{A}^{-} \chi_{-}+\eps^{N+\frac32} \tfrac{\lambda_{N+\frac12}(t)}{f''(\pm1)}\right)
  \\
  \label{eq:DefvA}
  \ve_\mathrm{A}(x,t)           & =\zg \ve_\mathrm{A}^{\mathrm{in}}(x,t)+(1-\zg )\left(\ve_\mathrm{A}^{+}(x,t)
  \chi_+ +\ve_\mathrm{A}^-(x,t)\chi_-\right){- {\mathbf{N}}\bar{a}_\eps(t)},
  \\
  \label{eq:DefpA}
  p_\mathrm{A}(x,t)             & =\zg p_\mathrm{A}^{\mathrm{in}}(x,t)+(1-\zg )
  \left(p_\mathrm{A}^{+}(x,t)
  \chi_+ +p_\mathrm{A}^{-}(x,t)\chi_- \right),
\end{alignat}
where $\mathbf{N}\colon \Omega \to \R^2$ is a smooth vector field such that $\mathbf{N}|_{\partial \Omega}= \no_{\partial\Omega}$ and $\bar{a}_\eps\colon (0,T)\to \R$ is some suitable function related to the compatibility condition
    \begin{equation*}
      \int_\Omega \operatorname{div} \ve_\mathrm{A} \,\dif x\, = \int_{\partial\Omega} \no_{\partial\Omega}\cdot   \ve_\mathrm{A} \,d\sigma = 0.
    \end{equation*}
The approximation solutions are constructed from an outer and inner expansion. The underlying concept of the ansatzes is to derive an expression for a solution at a positive distance from the interface and also along the diffuse interface. The two expansions are then assembled into an approximate solution in the whole domain. The ansatzes of the outer expansions (denoted with $\pm$) are
\begin{equation}
  c_\mathrm{A}^\pm(x,t)  =\sum_{k=0}^{N+2}\eps^{k}c_{k}^{\pm}(x,t), \  \ \mathbf{v}_\mathrm{A}^\pm(x,t) =\sum_{k=0}^{N+2}\eps^{k}\mathbf{v}_{k}^{\pm}(x,t),\ \  p_\mathrm{A}^\pm(x,t) = \sum_{k=-1}^{N+2}\eps^{k}p_{k}^{\pm}(x,t)
\end{equation}
and for the mass term
\begin{equation}\label{eq:DeflambdaA}
  \lambda_{\mathrm{A}}(t)=\tl_{\mathrm{A}}(t)+ \eps^{N+\frac12} \lambda_{N+\frac12}(t),\quad   \tl_{\mathrm{A}}(t)=\sum_{k=0}^{N+1} \varepsilon^k \lambda_k(t)
\end{equation}
with an upper limit that is one less than the other quantities. Here $\lambda_k(t)$ and $\lambda_{N+\frac12}(t)$ are independent of $x\in\Omega$.
It will be crucial that we use the following ansatz
\begin{equation}\label{eq:innerexpan'}
  \begin{aligned}
    c_\mathrm{A}^{\mathrm{in}} (x,t)   & = \tc_\mathrm{A}^{\mathrm{in}}(\rho,s,t) +\eps^{N+\frac32} \theta_1(\rho)\lambda_{N+\frac12}(t)                                                                                \\
                                       & \ +\eps^{N-\frac12}\left( \theta_0'(\rho)+
    \eps\partial_\rho \hat{c}_1(\rho,x,t)+
    \eps^2 \partial_\rho \hat{c}_2(\rho,x,t)\right)h_{N+\frac12}(S(x,t),t),
    \\
    \ve_\mathrm{A}^{\mathrm{in}} (x,t) & =\tv_\mathrm{A}^{\mathrm{in}}(\rho,x,t) {+ \eps^{N+\frac12}\hat{\we}(\rho,x,t)}, \\
    p_\mathrm{A}^{\mathrm{in}} (x,t)   & =\tp_\mathrm{A}^{\mathrm{in}}(\rho,x,t){+\eps^{{N-\frac12}}\hat{q}(\rho,x,t)}
  \end{aligned}
\end{equation}
for the inner expansions of $c_\mathrm{A}$, $\ve_\mathrm{A}$, and $p_\mathrm{A}$,
Here $\rho=\rho_\eps$ is as in \eqref{eq:StretchedVariable} with
\begin{equation}\label{eq:heps}
  h_\eps(s,t):= \sum_{k=0}^N \eps^k h_{k+1}(s,t)\qquad \text{for all }s\in \T^1, t\in [0,T_0].
\end{equation}
Moreover, a function $h_{N+\frac12}\colon \T^1\times [0,T_0]\to \R$, which may depend on $\eps$, is used such that $h_{N+\frac12}\in X_T$ is bounded (with respect to $\eps$),
where
\begin{equation}\label{eq:2.0}
  X_T=L^2(0,T;H^{5/2}(\T^1))\cap H^1(0,T;H^{1/2}(\T^1))
\end{equation}
normed by $$
  \|h\|_{X_T}\coloneqq\|h\|_{L^2(0,T;H^{5/2}(\T^1))}+\|h\|_{H^1(0,T;H^{1/2}(\T^1))}+ \|h|_{t=0}\|_{H^{3/2}(\T^1)}.$$
We note that $\lambda_{N+\frac12}$ will in the following determined in dependence of $h_{N+\frac12}$.
Moreover, $\theta_1$ is as in \eqref{eq:theta1}. We note that $\lim_{\rho \to \pm \infty}\theta_1(\rho)= \frac1{f''(\pm1)}$.

 Furthermore, the following standard ansatz for the first terms $\tc_\mathrm{A}^{\mathrm{in}}(\rho,x,t)$, $\tv_\mathrm{A}^{\mathrm{in}}(\rho,x,t)$ and $\tp_\mathrm{A}^{\mathrm{in}}(\rho,x,t)$ is used:
\begin{align}\label{eq:AnsatzVin}
  \tc_\mathrm{A}^{\mathrm{in}}(\rho,x,t) & = \sum_{k=0}^{N+2}\eps^k \tc_k(\rho,x,t), \\
  \tv_\mathrm{A}^{\mathrm{in}}(\rho,x,t) & = \sum_{k=0}^{N+2}\eps^k \tv_k(\rho,x,t), \\
  \tp_\mathrm{A}^{\mathrm{in}}(\rho,x,t) & =\sum_{k=-1}^{N+1}\eps^k \tp_k(\rho,x,t).
\end{align}
The functions $\tc_k$, $\tv_k$, and $\tp_k$ are smooth and maintain independence from $\eps$. The terms $\hat{\we}$ and $\hat{q}$ will be determined at a later stage, with $\hat{\we}$ satisfying the constraint $\no\cdot \hat{\we}=0$ throughout $\Gamma(2\delta)$. The expressions $\tc_\mathrm{A}^{\mathrm{in}}$, $\tv_\mathrm{A}^{\mathrm{in}}$, and $\tp^{\mathrm{in}}_\mathrm{A}$ follow the formulation in \eqref{eq:innerexpan'}, but with the condition $(h_{N+\frac12}, \hat{\we}, \hat{q})$ set to zero. Similarly, $(\tc_\mathrm{A}, \tv_\mathrm{A}, \tp_\mathrm{A}, \hat{\lambda}_\mathrm{A})$ represent the corresponding approximate solutions in $\Omega$ under the same zero condition for $(h_{N+\frac12}, \hat{\we}, \hat{q})$.
We note that the proposed ansatz of $c_\mathrm{A}^\mathrm{in}$ differs from the preceding work \cite{AbeFei23}, where $\hat{c}_1 \equiv 0$, but in this work we need $\hat{c}_1  \not\equiv 0$ as in \cite{ChenHilhorstLogak}.  
Because of this new term, the ansatz and the analysis of the error terms need a significant adaptation compared to the non-mass-conserving case.

The structure of the proof is as follows: First we construct $(\tc_\mathrm{A}, \tv_\mathrm{A}, \tp_\mathrm{A}, \hat{\lambda}_\mathrm{A})$ together with $h_\eps$, which are defined as in \eqref{eq:DefcA}-\eqref{eq:DefpA} without the terms depending on $h_{N+\frac12}$ and $\lambda_{N+\frac12}$, and \eqref{eq:DeflambdaA}, such that they satisfy \eqref{eq:ApproxS1}-\eqref{eq:ApproxS5} without the term $\eps^{N+\frac12}\zg \ue|_{\Gamma} \cdot \nabla c_\mathrm{A}$ and suitable estimates for the remainder terms, cf.\ Theorem~\ref{thm:Approx1} below. 
Afterwards, we determine $h_{N+\frac12}(S(x, t), t)$ and $\lambda_{N+\frac12}$ and its associated terms
        such that they yield $\eps^{N+\frac12}\ue|_{\Gamma} \cdot \nabla c_\mathrm{A}^{\mathrm{in}}$ (in equation \eqref{eq:approxAC}) up to lower-order terms, as detailed in Theorem~\ref{thm: h_N+1/2}. Here $\lambda_{N+\frac12}$ is chosen to ensure mass-conservation upto an error of order $O(\eps^{N+\frac32})$.
Finally, we select $\hat{\we}$ and $\hat{q}$ to eliminate the principal error term in the Navier-Stokes equation's right-hand side, ensuring all additional terms contribute only $O(\eps^{N+\frac12})$ in $L^2(0,T;(H^1(\Omega)^2)')$.
The initial step follows standard methodology. The result is summarized as follows:
\begin{thm}\label{thm:Approx1}
  For any $N\in \N$ there are smooth $(\tc_\mathrm{A}^{\mathrm{in}},\tv_\mathrm{A}^{\mathrm{in}},\tp_\mathrm{A}^{\mathrm{in}})$ defined in $\Gamma(2\delta)$ together with $h_\eps$ as in \eqref{eq:heps}, $\tl_\mathrm{A}$ defined in $[0,T_0]$, and $(c_\mathrm{A}^\pm, \ve_\mathrm{A}^\pm, p_\mathrm{A}^\pm)$  defined on $\Omega^\pm$, which are smooth and are approximate solutions in the following sense:
  \begin{enumerate}
    \item \emph{Inner expansion:} In $\Gamma(2\delta)$ we have
          \begin{alignat}{2}\nonumber
            \partial_t \tv_\mathrm{A}^{\mathrm{in}}\NSt{+ \tv_\mathrm{A}^{\mathrm{in}}\cdot \nabla \tv_\mathrm{A}^{\mathrm{in}}} -\Div (2\nu(\tc_\mathrm{A}^{\mathrm{in}})D \tv_\mathrm{A}^{\mathrm{in}}) +\nabla \tp_\mathrm{A}^{\mathrm{in}} & = -\eps \Div (\nabla \tc_\mathrm{A}^{\mathrm{in}}\otimes \nabla \tc_\mathrm{A}^{\mathrm{in}}) + R_\eps^{\mathrm{in}},
            \\\nonumber
            \Div \tv_\mathrm{A}^{\mathrm{in}}                                                                                                                                                                                                                                      & =\,   
            G_\eps^{\mathrm{in}},                                                                                                                                                                                                                                                                                                                                                                                                                              \\ \label{eq: remainder inner approx}
            \partial_t \tc_\mathrm{A}^{\mathrm{in}} + \tv_\mathrm{A}^{\mathrm{in}}\cdot \nabla \tc_\mathrm{A}^{\mathrm{in}}                                                                                                                                            & = \Delta \tc_\mathrm{A}^{\mathrm{in}} -\frac1{\eps^2} f'(\tc_\mathrm{A}^{\mathrm{in}})+\frac{1}{\varepsilon}\tl_\mathrm{A}+ S_\eps^{\mathrm{in}},
          \end{alignat}
          where
          \begin{alignat}{2}\label{eq:RemainderApprox1}
            \|(R_\eps^{\mathrm{in}}, \partial_t G_\eps^{\mathrm{in}}, S_\eps^{\mathrm{in}})\|_{L^\infty((0,T_0)\times\Omega)} & \leq C\eps^{N+1}, \\\label{eq:RemainderApprox2}
            \|G_\eps^{\mathrm{in}} \|_{L^\infty((0,T_0)\times\Omega)}                              & \leq C\eps^{N+2}.
          \end{alignat}
    \item \emph{Outer expansion:}  In $\Omega^\pm$ we have
          \begin{alignat}{2}
            \partial_t \ve_\mathrm{A}^\pm\NSt{+ \ve_\mathrm{A}^\pm\cdot \nabla \ve_\mathrm{A}^\pm} -\nu^\pm \Delta \ve_\mathrm{A}^\pm +\nabla p_\mathrm{A}^\pm & = R^\pm_\eps,
            \\\nonumber
            \Div \ve_\mathrm{A}^\pm                                                                                                                            & =\, 0,                                \\
            \ve_\mathrm{A}^\pm{|_{\partial\Omega}}                                                                                                             & =\, \ol{a}_\eps \no_{\partial\Omega},
          \end{alignat}
          where $\ol{a}_\eps\colon (0,T)\to \R$ is continuous and $ \|R_\eps^\pm \|_{L^\infty(\Omega\times [0,T_0])}\leq C\eps^{N+2}$  for all $\eps\in (0,1)$. For the concentration, we have for $t\in[0,T_0]$:
          \begin{equation*}
            \partial_t c_\mathrm{A}^{\pm}=-\frac{1}{\varepsilon^2}f'(c_\mathrm{A}^\pm)+\frac{1}{\varepsilon}\tl_\mathrm{A}(t)+{S}^\pm_\varepsilon \quad\text{ with $ \|S_\eps^\pm \|_{L^\infty(\Omega\times [0,T_0])}\leq C\eps^{N+1}$.}
          \end{equation*}
    \item \emph{Matching condition:} We have
          \begin{equation*}
            \begin{split}
                \| \p_x^{\beta}( \tv_\mathrm{A}^{\mathrm{in}}-\ve_\mathrm{A}^{+}
              \chi_+-\ve_\mathrm{A}^{-}\chi_-)\|_{L^\infty(\Gamma(2\delta)\setminus \Gamma(\delta) )}&\leq C e^{-\frac{\alpha\delta}{2\eps}}, \\
                \| \p_x^{\beta}(  \tp_\mathrm{A}^{\mathrm{in}}-p_\mathrm{A}^{+}
              \chi_+-p_\mathrm{A}^{-}\chi_-)\|_{L^\infty(\Gamma(2\delta)\setminus \Gamma(\delta) )}&\leq C e^{-\frac{\alpha\delta}{2\eps}},   \\
               \|   \p_x^{\beta}(\tc_\mathrm{A}^{\mathrm{in}}-c_\mathrm{A}^{+}
              \chi_+-c_\mathrm{A}^{-}\chi_-)\|_{L^\infty(\Gamma(2\delta)\setminus \Gamma(\delta) )}&\leq C e^{-\frac{\alpha\delta}{2\eps}}
            \end{split}
          \end{equation*}
          for all $\eps \in (0,1)$ and $\beta\in \N_0^n$.
          \item \emph{Approximate conservation of mass:} If $\tc_\mathrm{A}$ is defined as in \eqref{eq:DefcA} with $h_{N+\frac12}\equiv \lambda_{N+\frac12}\equiv0$, we have
          \begin{equation}\label{eq:meantcEstim}
            \int_{\Omega} \partial_t \tc_\mathrm{A}^\mathrm{in}(x,t) \dif{x}={O}(\varepsilon^{N+2})\quad \text{in }L^\infty(0,T_0).
          \end{equation}

  \end{enumerate}
\end{thm}
\begin{proof}
  Details of the proof are given in the appendix.
\end{proof}

\begin{rem}\label{rem: inner approx}Let us denote the errors of the inner expansions by $u_\mathrm{A}^{\mathrm{in}}\coloneqq c_\mathrm{A}^{\mathrm{in}}-\tc_\mathrm{A}^{\mathrm{in}}$ and $\mathbf{w}_\mathrm{A}^{\mathrm{in}}\coloneqq\mathbf{v}_\mathrm{A}^{\mathrm{in}}-\tv_\mathrm{A}^{\mathrm{in}}$. Then 
  \begin{align}
     & \partial_{t} c_\mathrm{A}^{\mathrm{in}}+\mathbf{v}_\mathrm{A}^{\mathrm{in}} \cdot \nabla c_\mathrm{A}^{\mathrm{in}}-\Delta c_\mathrm{A}^{\mathrm{in}}+\frac{1}{\varepsilon^{2}} f^{\prime}\left(c_\mathrm{A}^{\mathrm{in}}\right) - \varepsilon^{-1}\lambda_\mathrm{A}(t)\nonumber                                                                                        \\
     & =\partial_{t} \tc_\mathrm{A}^{\mathrm{in}}+\tv_\mathrm{A}^{\mathrm{in}} \cdot \nabla \tc_\mathrm{A}^{\mathrm{in}}-\Delta \tc_\mathrm{A}^{\mathrm{in}}+\frac{1}{\varepsilon^{2}} f^{\prime}\left(\tc_\mathrm{A}^{\mathrm{in}}\right)-\varepsilon^{-1}\lambda_\mathrm{A}(t)+\partial_{t} u_\mathrm{A}^{\mathrm{in}}\nonumber         \\
     & \quad+\mathbf{v}_\mathrm{A}^{\mathrm{in}} \cdot \nabla u_\mathrm{A}^{\mathrm{in}}+\mathbf{w}_\mathrm{A}^{\mathrm{in}} \cdot \nabla \tc_\mathrm{A}^{\mathrm{in}}-\Delta u_\mathrm{A}^{\mathrm{in}}+\frac{1}{\varepsilon^{2}} f^{\prime \prime}\left(\tc_\mathrm{A}^{\mathrm{in}}\right) u_\mathrm{A}^{\mathrm{in}}+\tilde{s}_\mathrm{A}^{\varepsilon}\nonumber \\
     & =\partial_{t} u_\mathrm{A}^{\mathrm{in}}+\mathbf{v}_\mathrm{A}^{\mathrm{in}} \cdot \nabla u_\mathrm{A}^{\mathrm{in}}-\Delta u_\mathrm{A}^{\mathrm{in}}+\frac{1}{\varepsilon^{2}} f^{\prime \prime}\left(\tc_\mathrm{A}^{\mathrm{in}}\right) u_\mathrm{A}^{\mathrm{in}}\label{eq: for h thm}
    +S_{\varepsilon}^{\mathrm{in}}+\tilde{s}_\mathrm{A}^{\varepsilon}+\mathbf{w}_\mathrm{A}^{\mathrm{in}} \cdot \nabla \tc_\mathrm{A}^{\mathrm{in}}
  \end{align}
  in $\Gamma_{t}(2 \delta), t \in\left[0, T_{\varepsilon}\right]$. Here $\tilde{s}_\mathrm{A}^{\varepsilon}$ contains terms that are quadratic in $u_\mathrm{A}^{\mathrm{in}}$ times $\varepsilon^{-2}$ and $S_\varepsilon^\mathrm{in}$ is as in \eqref{eq: remainder inner approx}.
 $\tilde{s}_\mathrm{A}^{\varepsilon}$ has the following properties:
          \begin{enumerate}
            \item $\tilde{s}_\mathrm{A}^{\varepsilon}\in{O}(\varepsilon^{2N-3+\frac{1}{2}}) = {O}(\varepsilon^{N+\frac{1}{2}})$ in $\mathrm{L}^{\infty}(0,T_{\varepsilon};\mathrm{L}^{2}(\Gamma_{t}(2\delta)))$ for $N \geq 3$,
            \item $\tilde{s}_\mathrm{A}^{\varepsilon}\in{O}(\varepsilon^{N+1})$ in $\mathrm{L}^{\infty}(0,T_{\varepsilon};(\mathrm{X}_{\varepsilon})^{\prime})$. 
            \end{enumerate}
            These claims follow from \eqref{eq:RemEstim1}-\eqref{eq:RemEstim2}, with uniform constants when $\|h_{N+\frac{1}{2}}\|_{\mathrm{X}_{T_{\varepsilon}}} \leq M$ for some $M>0$.
Moreover, $\mathbf{w}_\mathrm{A}^{\mathrm{in}} \cdot \nabla \tc_\mathrm{A}^{\text{in}}$ is ${O}(\varepsilon^{N+1})$ in $\mathrm{L}^{\infty}(0,T_{\varepsilon};\mathrm{L}^{2}(\Gamma_{t}(2\delta)))$. This follows directly from Lemma \ref{lem:rescale}, noting that $\mathbf{n} \cdot \hat{\mathbf{w}}(\rho,x,t)=0$ by the previous construction.
\end{rem}

For the following, we denote
\begin{equation}\label{eq:uA}
    u_{\mathrm A}(x, t)  =\zeta_{\Gamma} u^{\text{in}}(x, t)+(1-\zeta_{\Gamma})\eps^{N+\frac32} \frac{\lambda_{N+\frac12}}{f''(\pm1)}.
  \end{equation}
  Then $c_{\mathrm A}= \tc_{\mathrm A}+ u_{\mathrm A}$.
An essential estimate for the first terms in the latter remark is obtained in:
\begin{thm}\label{thm: h_N+1/2}
  For given $T_{\varepsilon} \in\left(0, T_{0}\right]$, $M>0$, $\mathbf{u}$ be as in Theorem~\ref{thm:approx}, 
 let $h\equiv h_{N+\frac{1}{2}} \in \mathrm{X}_{T_{\varepsilon}}$ be the solution of
  \begin{equation}\label{eq:PDE of h}
    \begin{aligned}
      D_t h+X_0^*(\ve) \cdot \nabla_{\Gamma} h-\Delta_{\Gamma} h+X_0^*\left(g_0\right) h & =-X_0^*(\mathbf{n}\cdot \mathbf{u})+\frac2{\sigma}\lambda_{N+\frac12}  &  & \text { on } \mathbb{T}^1 \times\left[0, T_{\varepsilon}\right], \\
      \left.h\right|_{t=0}                                                                                                                      & =0                                   &  & \text { on } \mathbb{T}^{1}
    \end{aligned}
  \end{equation}
  where $\lambda_{N+\frac12}\colon (0,T_\eps)\to \R$ is chosen such that  
  \begin{equation}\label{eq:lambdaN12}
    \int_{\T^1}\left(V(s,t)\kappa(s,t) h_{N+\frac12}(s,t)-D_t h_{N+\frac12}(s,t)\right) \dif{s} =0,
  \end{equation}
  and $g_0$ is as in \eqref{eq:g0} below, and
  \begin{equation*}
    R_u\coloneqq \partial_t u_\mathrm{A}+\mathbf{v}_\mathrm{A} \cdot \nabla u_\mathrm{A}-\Delta u_\mathrm{A}+\frac{1}{\varepsilon^2} f^{\prime \prime}\left(\tc_\mathrm{A}\right) u_\mathrm{A}-\eps^{N-\frac12}\lambda_{N+\frac12} +\varepsilon^{N+\frac{1}{2}} \left.\zg\mathbf{u}\right|_{\Gamma} \cdot \nabla c_\mathrm{A}.
  \end{equation*}
 Then, if $\left\|h_{N+\frac{1}{2}}\right\|_{\mathrm{X}_{T_{\varepsilon}}} \leq M$ for some $M>0$, there is some $C(M)>0$ such that for $p=1,2$ 
 \begin{alignat}{1}\nonumber
   \left\|R_u\right\|_{\mathrm{L}^2\left(0, T_{\varepsilon} ; X_{\varepsilon}^{\prime}\right)}+  \left\|R_u\right\|_{\mathrm{L}^2\left(\Omega\times (0, T_{\varepsilon})\setminus \Gamma(2\delta) \right)} &\leq C(M) \varepsilon^{N+1},                \\\nonumber
   \left\|R_u\right\|_{\mathrm{L}^2\left(0, T_{\varepsilon} ; \mathrm{L}^p\left(\Gamma_t(2\delta)\right)\right)} &\leq C(M) \varepsilon^{N-\frac12+ \frac1p},\\
\label{eq:MeanEstimu_A}
\sup_{0\leq t\le T_\eps}\left|\fint_{\Omega}\partial_t u_\mathrm{A}(x,t)\,dx\right|&\leq C(M)\eps^{N+\frac32}.
 \end{alignat}
\end{thm}
\begin{rem}
  By inserting \eqref{eq:PDE of h} in \eqref{eq:lambdaN12}, we determine $\lambda_{N+\frac12}$:
  $$
  \lambda_{N+\frac{1}{2}}(t) = \frac{\sigma}{2}\int_{\mathbb{T}^1}\Bigl[V(s,t)\kappa(s,t)- \partial_s X_0^\ast(\nabla S\cdot \ve_0)(s,t) + X_0^*(g_0)(s,t)\Bigr]h_{N+\frac{1}{2}}(s,t)+X_0^*(\mathbf{n}\cdot\mathbf{u})(s,t)\,\mathrm{d}s,
  $$
where $|\mathbb{T}^1|=1$ (cf. Section \ref{sec:Prelim}).
Hence \eqref{eq:PDE of h} yields a parabolic partial differential equation for $h$ with an additional non-local term $\int_{\T^1} k(s,t) h(s,t)\dif{s} $ for a suitable smooth function $k$. Solvability of this equation (with $h|_{t=0}=0$) can be shown in the same way as in \cite[Theorem~2.9]{StokesAllenCahn} since the additional nonlocal term is of lower order. Moreover, we have
  \begin{equation*}
    \|h\|_{X_{T_\eps}}\leq C(T_0) \|\ue|_{\Gamma}\|_{L^2(0,T_\eps; L^2(\Gamma_t))}
  \end{equation*}
 for some $C(T_0)>0$ independent of $\ue, T_\eps, \eps$  by the open mapping theorem and an extension of $\ue|_{\Gamma}(t)$ by zero for $t> T_\eps$.
\end{rem}  
The following proof is a modification of the one in \cite[Theorem 3.1]{AbeFei23} and the essential
difference in our derivation is the new term $c_1$ and the mass term $\lambda_0$.

\medskip

\noindent
\begin{proof*}{of Theorem \ref{thm: h_N+1/2}}
  \emph{Step 1: Ansatz and interpretation of $h$.}
  Let $u_{\mathrm{A}}$ be as in \eqref{eq:uA}. The evolution equation \eqref{eq:PDE of h} in the theorem is a parabolic equation on evolving hypersurfaces and it has the unique solution $h\in X_{T_\varepsilon}$ according to Lemma~2.10 in \cite{StokesAllenCahn}.
  The function $h$ describes a normal shift of the interface at order $\varepsilon^{N+\frac12}$, and the corresponding perturbation of the inner profile lies in the approximate kernel direction $\theta_0'$ of the linearized Allen-Cahn operator.
  By choosing $h$ via \eqref{eq:PDE of h}, we enforce the solvability condition in this kernel direction so that the leading-order inner error produces precisely the additional transport term $\varepsilon^{N+\frac12}\mathbf{u}|_\Gamma\cdot\nabla c_\mathrm{A}^\mathrm{in}$ up to lower-order remainders.
  We will calculate all the derivatives ocurring in the mass-conserving Allen-Cahn equation \eqref{eq:NSAC3} step by step and then identify the equation for $h$ as the solvability condition in the $\theta_0'$–direction.

  \smallskip
  
  \noindent  \emph{Step 2: Expansion and order-by-order cancellations.}
  Using Lemma~\ref{lem:ChainRule}, Lemma~\ref{lem:rescale} and the expansion of $h_\varepsilon$ we obtain
    \begin{align}\nonumber
     & \partial_t u_{\mathrm{A}}= 
      \varepsilon^{N-\frac{3}{2}}\zeta_\Gamma\partial_t d_\Gamma \theta_0^{\prime \prime}(\rho) h_{N+\frac{1}{2}}(s, t)-\varepsilon^{N-\frac{1}{2}}\zeta_\Gamma \theta_0^{\prime \prime}(\rho)
      \partial_t^{\Gamma} h_1(r, s, t) h_{N+\frac{1}{2}}(s, t)\\\label{eq:Expdtu}
   &  \ +\varepsilon^{N-\frac{1}{2}}\zeta_\Gamma \theta_0^{\prime}(\rho)
      \partial_t^{\Gamma} h_{N+\frac{1}{2}}(r, s, t)
      +\varepsilon^{N-\frac{1}{2}}\zeta_\Gamma \partial_t d_\Gamma \partial_\rho^2\hat{c}_1(\rho,x,t)h_{N+\frac12}(s,t)
      +O\left(\varepsilon^{N+\frac12+\frac1p}\right)
    \end{align}
  in $\mathrm{L}^{2}\left(0, T_{\varepsilon} ;L^p(\Omega)\right)$, $p=1,2$, where $s=S(x, t), r=d_\Gamma(x, t), \rho=\rho(x, t)$ as before. In a similar manner, we derive
  \[
    \begin{aligned}
      \mathbf v_A\cdot\nabla u_{\mathrm{A}}
      = & \varepsilon^{N-\frac{3}{2}}\zeta_\Gamma \mathbf{n}_{\Gamma_t}\cdot \mathbf{v}_0
      \theta_0^{\prime \prime}(\rho) h_{N+\frac{1}{2}}(s, t)                 \\
        & -\varepsilon^{N-\frac{1}{2}}\zeta_\Gamma \mathbf{v}_0 \cdot \nabla^{\Gamma} h_1(r, s, t)
      \theta_0^{\prime \prime}(\rho) h_{N+\frac{1}{2}}
      +\varepsilon^{N-\frac{1}{2}}\zeta_\Gamma \mathbf{v}_0 \cdot \nabla^{\Gamma} h_{N+\frac{1}{2}}(r, s, t)
      \theta_0^{\prime}(\rho)                                                          \\
        & +\varepsilon^{N-\frac{1}{2}}\zeta_\Gamma\mathbf{v}_0\cdot \mathbf{n}_{\Gamma_t}
      \partial_\rho^2 \hat{c}_1(\rho, x,t) h_{N+\frac 12}(s,t)
      +O\left(\varepsilon^{N+1}\right)
    \end{aligned}
  \]
  in $\mathrm{L}^{2}\left(0, T_{\varepsilon} ;L^2(\Omega)\right)$ as well as
  \[
    \begin{aligned}
      \Delta u_{\mathrm{A}}(x,t)
      = & \varepsilon^{N-\frac{5}{2}}\zeta_\Gamma \theta_0^{\prime \prime \prime}(\rho) h_{N+\frac{1}{2}}(s, t)
      +\varepsilon^{N-\frac{3}{2}}\zeta_\Gamma \Delta d_\Gamma \theta_0^{\prime \prime}(\rho) h_{N+\frac{1}{2}}(s, t)+\varepsilon^{N-\frac{3}{2}}\zeta_\Gamma\partial_\rho^3\hat{c}_1(\rho,x,t) h_{N+\frac{1}{2}}                  \\
          & +\left|\nabla^{\Gamma} h_1(r, s, t)\right|^2 \varepsilon^{N-\frac{1}{2}}\zeta_\Gamma
      \theta_0^{\prime \prime \prime}(\rho) h_{N+\frac{1}{2}}(s, t)
         + \varepsilon^{N-\frac{1}{2}}\zeta_\Gamma\partial_\rho^2\theta_1(\rho)\lambda_{N+\frac12}
        \\
          &+\varepsilon^{N-\frac{1}{2}} \zeta_\Gamma\partial_\rho^3 \hat{c}_2(\rho, x, t) h_{N+\frac{1}{2}}(s, t)    +\varepsilon^{N-\frac{1}{2}} \Delta^{\Gamma} h_{N+\frac{1}{2}} \theta_0^{\prime}(\rho)       \\
          & -\varepsilon^{N-\frac{1}{2}}\zeta_\Gamma\Big(h_{N+\frac{1}{2}} \Delta^{\Gamma} h_1
      +2 \nabla^{\Gamma} h_1 \cdot \nabla^{\Gamma} h_{N+\frac{1}{2}}\Big) \theta_0^{\prime \prime}(\rho)                               \\
          & +\varepsilon^{N-\frac12}\zeta_\Gamma\Delta d_\Gamma (x,t)\partial_\rho^2 \hat{c}_1(\rho,x,t)
      h_{N+\frac12}(s,t) +O\left(\varepsilon^{N+1}\right)
    \end{aligned}
  \]
  in $\mathrm{L}^{2}\left(0, T_{\varepsilon} ;L^2(\Omega)\right)$.
  For the potential term, we calculate the following:
  \begin{align}\nonumber
     & \frac{1}{\varepsilon^2} f^{\prime \prime}\left(\tc_\mathrm{A}\right) u_\mathrm{A}- \eps^{N-\frac12}\lambda_{N+\frac12}
    = \varepsilon^{N-\frac{5}{2}}\zeta_\Gamma f^{\prime \prime}\left(\theta_0(\rho)\right)
    \theta_0^{\prime}(\rho) h_{N+\frac{1}{2}}(s, t)                                                                                                                   \\\nonumber
     & \quad     +\varepsilon^{N-\frac{3}{2}}\zeta_\Gamma\Big[
      f^{\prime \prime \prime}\left(\theta_0(\rho)\right)\theta_0'(\rho) \hat{c}_1(\rho, x, t)
    +f^{\prime \prime}\left(\theta_0(\rho)\right) \partial_\rho \hat{c}_1(\rho, x, t)\Big] h_{N+\frac{1}{2}}(s, t)                                                    \\\nonumber
     & \quad +\varepsilon^{N-\frac{1}{2}}\zeta_\Gamma\Big[
      f^{\prime\prime\prime}(\theta_0(\rho))\hat{c}_2(\cdot)\theta_0^{\prime}(\rho)
    +f^{\prime\prime}(\theta_0(\rho))\partial_\rho\hat{c}_2(\cdot)+ f^{\prime \prime\prime}\left(\theta_0(\rho)\right)\hat{c}_1(\cdot) \partial_\rho \hat{c}_1(\cdot) \\
     & \qquad\quad\quad\nonumber
      +\tfrac12 f^{(4)}(\theta_0(\rho))(\hat{c}_1(\cdot)^2\theta_0^{\prime}(\rho)
    \Big]h_{N+\frac{1}{2}}(s,t)+ \eps^{N-\frac12}\left(f''(\theta_0)\theta_1(\rho)-1\right)\lambda_{N+\frac12} +{O}\left(\varepsilon^{N+1}\right)\label{eq: potential}
  \end{align}
  in $\mathrm{L}^{2}\left(0, T_{\varepsilon} ;L^2(\Omega)\right)$.
  The first square bracket of order $\varepsilon^{N-3/2}$ can be simplified using the ODE determining $\hat{c}_1$,
  \begin{equation}\label{eq:2nd derivative of c1}
    \partial_\rho^2\hat{c}_1-f''(\theta_0)\hat{c}_1
    =\theta_0'(\rho)(\partial_td_\Gamma+\mathbf{v}_0\cdot\nabla d_\Gamma-\Delta d_\Gamma +g_0 d_\Gamma)
    -\lambda_0(t)
    =-(1-2{\sigma}^{-1}\theta_0'(\rho))\lambda_0(t).
  \end{equation}
  Differentiation with respect to $\rho$ yields
  \begin{equation}
    \label{eq:third derivative of c1 eq}
    f^{\prime \prime}\left(\theta_0(\rho)\right) \partial_\rho \hat{c}_1
    +f^{\prime \prime \prime}\left(\theta_0(\rho)\right) \theta_0^{\prime}(\rho) \hat{c}_1
    =\partial_\rho^3 \hat{c}_1-2\sigma^{-1}\lambda_0(t)\theta_0''(\rho).
  \end{equation}
  The second square bracket of order $\varepsilon^{N-1/2}$ is simplified analogously using the ODE for $\hat{c}_2$ in \eqref{eq: c2 ode simplified}:
  \[
    \begin{aligned}
       & f^{\prime\prime\prime}(\theta_0(\rho)) \hat{c}_2(\rho,x,t)\,\theta_0'(\rho)
      + f^{\prime\prime}(\theta_0(\rho))\partial_\rho\hat{c}_2(\rho,x,t)+ f^{\prime\prime\prime}(\theta_0(\rho))\hat{c}_1(\rho,x,t)\partial_\rho\hat{c}_1(\rho,x,t)
      \\&\quad+\tfrac12 f^{(4)}(\theta_0(\rho))(\hat{c}_1(\rho,x,t))^2  \theta_0^{\prime}(\rho)              = \partial_\rho^3\hat{c}_2(\rho,x,t)
      +2\sigma^{-1}\lambda_1(t)\theta_0''(\rho)+\theta_0'''(\rho)|\nabla^\Gamma h_1| ^2    
        -  \partial_\rho \tilde{B}^1.
    \end{aligned}
  \]
  Altogether, the potential term becomes the following:
  \[
    \begin{aligned}
      &\frac{1}{\varepsilon^2} f^{\prime \prime}\left(\tc_\mathrm{A}\right) u_\mathrm{A}-\eps^{N-\frac12}\lambda_{N+\frac12}
      \\
      &\quad =\varepsilon^{N-\frac{5}{2}} f^{\prime \prime}\left(\theta_0(\rho)\right)
      \theta_0^{\prime}(\rho) h_{N+\frac{1}{2}}(s, t)+\varepsilon^{N-\frac{3}{2}}\Big[
        \partial_\rho^3 \hat{c}_1(\rho,x,t)
        -2\sigma^{-1}\lambda_0(t)\theta_0''(\rho)
      \Big] h_{N+\frac{1}{2}}(s, t)                                                  \\
          &\qquad +\varepsilon^{N-\frac{1}{2}}\Big[
        \partial_\rho^3\hat{c}_2(\rho,x,t)
        -2\sigma^{-1}\lambda_1(t)\theta_0''(\rho)+\theta_0'''(\rho)|\nabla^\Gamma h_1| ^2    
        -  \partial_\rho \tilde{B}^1
        \Big]h_{N+\frac{1}{2}}(s,t)                                                \\&
       \qquad+\eps^{N-\frac12}\left(f''(\theta_0)\theta_1(\rho)-1\right)\lambda_{N+\frac12} +{O}\left(\varepsilon^{N+1}\right)
    \end{aligned}
  \]
  in $\mathrm{L}^{2}\left(0, T_{\varepsilon} ;L^2(\Omega)\right)$, where
  \begin{equation*}
  \partial_\rho \tilde{B^1}=\partial_\rho^2 c_1(\rho)(\partial_t d_\Gamma + \mathbf{v}_0\cdot \nabla d_\Gamma+\Delta d_\Gamma-g_0(\rho\theta_0''(\rho)+\theta_0')).  
  \end{equation*}
Collecting all contributions, we obtain
  \begin{align*}
    &\partial_t u_\mathrm{A} + \mathbf{v}_\mathrm{A} \cdot \nabla u_\mathrm{A}
    - \Delta u_\mathrm{A}
    + \frac{1}{\varepsilon^2} f''(\tc_\mathrm{A}) u_\mathrm{A}-\eps^{N-\frac12}\lambda_{N+\frac12}\\
    &\quad =\varepsilon^{N-\frac52}\zeta_\Gamma\mathcal{ R}_{-\frac52}
    +\varepsilon^{N-\frac32}\zeta_\Gamma\mathcal{ R}_{-\frac32}
    +\varepsilon^{N-\frac12}\zeta_\Gamma\mathcal{ R}_{-\frac12}
    +{O}(\varepsilon^{N+1})
  \end{align*}
  in $\mathrm{L}^{2}\left(0, T_{\varepsilon} ;L^2(\Omega)\right)$, where
  \[
    \mathcal R_{-\frac52}(\rho,x,t)
    \coloneqq  \big(-\theta_0'''(\rho)+f''(\theta_0(\rho))\theta_0'(\rho)\big)h_{N+\frac{1}{2}}
  \]
  vanishes due to the optimal profile equation \eqref{eq: optimal profile}, and
  \[
  \begin{aligned}
        &\mathcal R_{-\frac32}(\rho,x,t)
    \coloneqq  \underbrace{\Big(\partial_t d_{\Gamma_t}+\mathbf{n}_{\Gamma_t}\cdot\mathbf{v}_{0}
    -\Delta d_\Gamma-2\sigma^{-1}\lambda_0(t)\Big)}_{=-g_0d_\Gamma}\theta_0''(\rho)h_{N+\frac12}\\
    &\quad = -g_0 \varepsilon (\rho+h_\varepsilon)\theta_0''(\rho)h_{N+\frac12}
    = -g_0 \varepsilon(\rho+h_1)\theta_0''(\rho)h_{N+\frac12}+ O(\varepsilon^{N+1}) \text{ in }\mathrm{L}^{2}\left(0, T_{\varepsilon} ;X_\varepsilon'\right)
  \end{aligned}
  \]
  increases its $\varepsilon$ order to $\mathcal{R}_{-1/2}$. The $h_1$-term becomes zero when tested with $\theta_0'$ and the other one with $-g_0\rho\theta_0''$ cancels out with the opposite term in \eqref{eq:R_-1/2}, in particular $\partial_\rho \tilde{B}^1$ in \eqref{eq:B1 tilde}.

  \medskip
  
  \noindent\emph{Step 3: Solvability condition in the $\theta_0'$–direction and remainder estimates.}
  For the remaining term (with $h\equiv h_{N+\frac{1}{2}}$) we have
  \begin{equation}\label{eq:R_-1/2}
    \begin{aligned}
      \mathcal R_{-\frac12}(\rho,x,t)
      \coloneqq  & \theta_0'(\rho)\Big(\partial_t^\Gamma h+\mathbf v_0\cdot\nabla^\Gamma h-\Delta^\Gamma h- \tfrac{2}\sigma\lambda_{N+\frac12}\Big) 
   \\                & +\theta_0''(\rho)\,h\Big(-\partial_t^\Gamma h_1-\mathbf v_0\cdot\nabla^\Gamma h_1
      +\Delta^\Gamma h_1\Big)
      +2\theta_0''(\rho)\,\nabla^\Gamma h_1\cdot\nabla^\Gamma h                                                              \\
                   & -\theta_0'''(\rho)\,h\,|\nabla^\Gamma h_1|^2 +\partial_\rho^2\hat c_1h\big(\partial_t d_{\Gamma_t}+\mathbf v_0\cdot\mathbf n_{\Gamma_t}
      - \Delta d_\Gamma\big)                                                                                                 \\
                   & +(2\sigma^{-1}\lambda_1(t)\theta_0''(\rho)+\theta_0'''(\rho)|\nabla^\Gamma h_1| ^2    
        -  \partial_\rho \tilde{B}^1)h
    \end{aligned}
  \end{equation}
  with
  \begin{equation}\label{eq:B1 tilde}
  \partial_\rho \tilde{B^1}=\partial_\rho^2 c_1(\rho)(\partial_t d_\Gamma + \mathbf{v}_0\cdot \nabla d_\Gamma-\Delta d_\Gamma)-g_0(\rho\theta_0''(\rho)+\theta_0'),
\end{equation}
where we have used \eqref{eq:theta1}.
  The first line in \eqref{eq:R_-1/2} on $\Gamma$, together with the term $\theta_0' g_0$ coming from $\partial_\rho \tilde{B}^1 $, is replaced by the evolution equation on the hypersurfaces
  \[
    \partial_t^{\Gamma} h_{N+\frac{1}{2}}
    +\mathbf{v}_0\cdot \nabla_{\Gamma} h_{N+\frac{1}{2}}
    -\Delta_\Gamma h_{N+\frac{1}{2}}
    +g_0 h_{N+\frac{1}{2}}-{\frac{2}{\sigma}}\lambda_{N+\frac12}
    =-\mathbf{n} \cdot \mathbf{u}|_{\Gamma}.
  \]
  which is exactly \eqref{eq:PDE of h} written on $\Gamma_t$ since
  \begin{equation*}
  \left.\varepsilon^{N+\frac{1}{2}}\zg \mathbf{u}\right|_{\Gamma} \cdot \nabla c_\mathrm{A}^\mathrm{in} = \left.\varepsilon^{N-\frac{1}{2}}\zg \mathbf{u}\right|_{\Gamma}\cdot \mathbf{n} \theta_0'(\rho)+ O(\eps^{N+1})
  \end{equation*}
  in $\mathrm{L}^{2}\left(0, T_{\varepsilon} ;L^2(\Omega)\right)$. Remember that a function 
$\varphi$ defined on~$\Gamma_t$ is identified with $\varphi \circ X_0(\cdot,t)$ 
on~$\mathbb{T}^1$ via the flow map $X_0 \colon \mathbb{T}^1 \times [0,T_\varepsilon] 
\to \Gamma_t$; the pullback $X_0^*$ makes this 
identification explicit.
  We will discuss that all the other lines in \eqref{eq:R_-1/2} are orthogonal to $\theta_0'$:
  By construction and oddness of $\theta_0$, all terms containing $\theta_0''$ vanish when tested against $\theta_0'$, the first summand in the third line of \eqref{eq:R_-1/2} cancels with its opposite from the last line, the second summand of the third line cancels its opposite in $\partial_\rho \tilde{B}^1$.
 Thus, all terms at order $\varepsilon^{N-1/2}$ can be collected in a remainder $\varepsilon^{N-1/2} a(\rho,x,t)$ with
  \[
    \int_{\mathbb{R}} a(\rho, s, t) \theta_0^{\prime}(\rho) \,d \rho=0
    \quad \text { for all } s \in \mathbb{T}^1,\, t \in\left[0, T_{\varepsilon}\right].
  \]
  We decompose $\varepsilon^{N-1/2}a(\rho,x,t)=f_1(x,t)+f_2(x,t)$ with
  \[
    f_1(x, t)=\varepsilon^{N-\frac{1}{2}} a(\rho, S(x, t), t),\qquad
    f_2(x, t)=\varepsilon^{N-\frac{1}{2}}(a(\rho, x, t)-a(\rho, S(x, t), t))+{O}\left(\varepsilon^{N+1}\right)
  \]
  in $\mathrm{L}^{2}\left(0, T_{\varepsilon} ;X_\varepsilon'\right)$.
  Because of \cite[Lemma~2.11]{AbeFei23} and Lemma~\ref{lem:rescale}, we get
  \[
    \left\|f_1\right\|_{\mathrm{L}^2\left(0, T_{\varepsilon} ; X_{\varepsilon}^{\prime}\right)}
    \leq C \varepsilon^{N+1},\qquad
    \left\|f_1\right\|_{\mathrm{L}^2\left(0, T_{\varepsilon} ; \mathrm{L}^p\left(\Gamma_t(2 \delta)\right)\right)}
    \leq C \varepsilon^{N-\frac12+ \frac1p},
  \]
  and, using \cite[Corollary~2]{StokesAllenCahn},
  \[
    \left\|f_2\right\|_{\mathrm{L}^2\left(0, T_{\varepsilon} ; X_{\varepsilon}^{\prime}\right)}
    \leq C\left\|f_2\right\|_{\mathrm{L}^2\left(0, T_{\varepsilon} ;
        \mathrm{L}^2\left(\Gamma_t(2 \delta)\right)\right)}
    \leq C' \varepsilon^{N+1}.
  \]
  Putting everything together, we obtain
  \[
    \begin{aligned}
      \partial_t u_\mathrm{A}
      +\mathbf{v}_\mathrm{A}\cdot \nabla u_\mathrm{A}
      -\Delta u_\mathrm{A}
      +\frac{1}{\varepsilon^2} f^{\prime \prime}\left(\tc_\mathrm{A}\right) u_\mathrm{A}
       & =-\varepsilon^{N-\frac{1}{2}} \mathbf{n} \cdot \mathbf{u}|_{\Gamma} \theta_0^{\prime}(\rho)
      +\varepsilon^{N-\frac{1}{2}} a(\rho, x, t)+{O}\left(\varepsilon^{N+1}\right)                        \\
       & =-\varepsilon^{N+\frac{1}{2}} \mathbf{u}|_{\Gamma} \cdot \nabla \tc_\mathrm{A}
      +\varepsilon^{N-\frac{1}{2}} a(\rho, x, t)+{O}\left(\varepsilon^{N+1}\right)
    \end{aligned}
  \]
  in $\mathrm{L}^{2}\left(0, T_{\varepsilon} ;X_\varepsilon'\right)$ and $\mathrm{L}^2(\Omega\times (0,T_\eps)\setminus \Gamma(2\delta))$, and the norm estimates for $f_1$ and $f_2$ yield precisely the bounds stated in Theorem~\ref{thm: h_N+1/2} for both $\mathrm{L}^2(0,T_\varepsilon;X_\varepsilon')$ and $\mathrm{L}^2(0,T_\varepsilon;\mathrm{L}^2(\Gamma_t(2\delta)))$.

  \medskip

  \noindent
  \emph{Step 4: Approximate mass conservation:}
  Finally, to show \eqref{eq:MeanEstimu_A} we apply Lemma~\ref{lem:MeanEstim} to \eqref{eq:Expdtu}. More precisely, we apply the first statement of Lemma~\ref{lem:MeanEstim} to the second and fourth term in \eqref{eq:Expdtu}. The second statement is applied to estimate the sum of the first and third term using~\eqref{eq:lambdaN12}.  
  This completes the proof.
\end{proof*}

%%%%%%%%%%%%%%%%%%%%%%%%%%%%%%%%%%%%%%%%%%%%%%%
\medskip
\noindent
\begin{proof*}{of Theorem~\ref{thm:approx}} The estimates for the part in the Navier-Stokes equations, i.e., $R_\eps^1,R_\eps^2, G_\eps$, are done as in the proof of \cite[Theorem 3.1]{AbeFei23}. The choice of $\we$ and $q$ is given there. The remaining task is to provide estimates related to the mass conservation in the Allen-Cahn equation.
  
  In order to accommodate the Allen-Cahn part, we first decompose the terms in
  $$
  \partial_t c_\mathrm{A} + \ve_\mathrm{A}\cdot \nabla c_\mathrm{A}+\eps^{N+\frac12}\zg\ue|_{\Gamma} \cdot \nabla c_\mathrm{A} = \Delta c_\mathrm{A} -\frac1{\eps^2} f'(c_\mathrm{A})+\frac{1}{\varepsilon}\lambda_\mathrm{A}
  + S_\eps
  $$
  into the part with and without $h_{N+1/2}$ after inserting the definition of $c_\mathrm{A}$ and $\mathbf{v}_\mathrm{A}$. As before we denote the terms without $h_{N+1/2}$ as $\tc_\mathrm{A}$ and $\tv_\mathrm{A}$ and the term including $h_{N+1/2}$ by $u_\mathrm{A}$. Then we have
  \begin{equation*}
    \begin{aligned}
      \partial_t c_\mathrm{A} & +\mathbf{v}_\mathrm{A}\cdot \nabla c_\mathrm{A}-\Delta c_\mathrm{A} +\varepsilon^{-2}f'(c_\mathrm{A})-\varepsilon^{-1}\lambda_\mathrm{A}+\varepsilon^{N+1/2}\mathbf{u}|_\Gamma\cdot\nabla c_\mathrm{A}                                                                                                                                                           \\
                              & =
      \partial_t \tc_\mathrm{A}+\tv_\mathrm{A}\cdot \nabla \tc_\mathrm{A}-\Delta \tc_\mathrm{A} +\varepsilon^{-2}f'(\tc_\mathrm{A})-\varepsilon^{-1}\hat\lambda_\mathrm{A}                                                                                                                                                                                    \\
                              &\quad +\partial_t u_\mathrm{A}+\mathbf{v}_\mathrm{A} \cdot \nabla u_\mathrm{A}-\Delta u_\mathrm{A}+\frac{1}{\varepsilon^2} f^{\prime \prime}\left(\tc_\mathrm{A}\right) u_\mathrm{A}-\eps^{N-\frac12}\lambda_{N+\frac12}+\varepsilon^{N+\frac{1}{2}}\left.\zg \mathbf{u}\right|_{\Gamma} \cdot \nabla c_\mathrm{A}.
    \end{aligned}
  \end{equation*}
  The middle line is of order ${O}(\varepsilon^{N+1})$ in $\mathrm{L}^\infty((0,T_\varepsilon)\times \Omega)$ by Theorem \ref{thm:Approx1} and the previous claim. Then the middle line is also of order ${O}(\varepsilon^{N+1})$ in $\mathrm{L}^2(0,T_\varepsilon;X_\varepsilon')$ since: Denote the middle line as $R_\varepsilon$, then
  \begin{align*}
    \|R_\varepsilon\|_{\mathrm{L}^2(0,T_\varepsilon;X_\varepsilon')}\leq C \sqrt{T_\varepsilon}\|R_\varepsilon\|_{\mathrm{L}^\infty(0,T_\varepsilon;\mathrm{L}^2(\Omega))} \leq C \sqrt{T_0|\Omega|}\|R_\varepsilon\|_{\mathrm{L}^\infty(0,T_\varepsilon;\mathrm{L}^\infty(\Omega))}.
  \end{align*}
  The last line is according to the estimates from Theorem \ref{thm: h_N+1/2} of order ${O}(\varepsilon^{N+1})$ in $\mathrm{L}^2(0,T_\varepsilon; X_\varepsilon')$. Overall, we have $S_\varepsilon = {O}(\varepsilon^{N+1})$ in $\mathrm{L}^2(0,T_\varepsilon;X_\varepsilon')$.

  Finally, \eqref{eq:ApproxS5} with the estimate of $\bar{S}_\eps$ follows from \eqref{eq:meantcEstim} and \eqref{eq:MeanEstimu_A}.
\end{proof*}

%%%%%%%%%%%%%%%%%%%%%%%%%%%%%%%%%%%%%%%%%%%%%%%%%%%%%%%%%%%%%%%%%%%%%%%%%%%%%%%%%%%%%%%%%%%%

\section{Proof of Main Result}\label{sec:main result}
\subsection{Preparations}\label{subsec:prep}

In this subsection, we work with two sets of solutions: $(c_\mathrm{A}, \ve_\mathrm{A}, p_\mathrm{A})$ and $(\tc_\mathrm{A},\tv_\mathrm{A}, \tp_\mathrm{A})$, as defined in Section~\ref{sec:ApproxSolutions}. While $(\tc_\mathrm{A},\tv_\mathrm{A}, \tp_\mathrm{A})$ are fully determined, $(c_\mathrm{A}, \ve_\mathrm{A}, p_\mathrm{A})$ remain depending on a parameter $\mathbf{u}$ that will be specified later. For notational simplicity, we will use $\no$ in place of $\no_\Gamma$ throughout this discussion.
Let $\tilde{\we} \colon \Om\times [0,T_0]\to \R^2$ be such that
\begin{equation}\label{velocity decompose}
  \ve_\eps= \ve_\mathrm{A}+ \tilde{\we}.
\end{equation}
Then 
\begin{equation}\label{eq:w}
  \begin{split}
    \partial_t \tilde{\we}{+ \ve_\eps \cdot \nabla \tilde{\we}}-\Div(2\nu(c_\eps) D\tilde{\we})&+\nabla q  =-\eps\Div (\nabla u\otimes^s \nabla c_\mathrm{A})-\eps\Div (\nabla u\otimes  \nabla u) \\&+\Div(2(\nu(c_\eps)-\nu(c_\mathrm{A})) D\ve_\mathrm{A}){- \tilde{\we}\cdot \nabla \ve_\mathrm{A}}- {R_\eps^1-R_\eps^2},\\
    \Div \tilde{\we}                                                                                     & = -G_\eps,                                                                              \\
    \tilde{\we}|_{t=0}                                                                                   & =\ve_{0,\eps} - \ve_{\mathrm{A},0},                                                     \\
    \tilde{\we}|_{\partial\Omega}                                                                        & = 0,
  \end{split}
\end{equation}
for a suitable pressue $q\colon \Om\times [0,T_0]\to \R$. Here $u=c_\eps-c_\mathrm{A}$ as before, $a\otimes^s b=a\otimes b+b\otimes a$ and $R_\eps, G_\eps$ are as in Theorem~\ref{thm:approx}.
For the following we choose
\begin{equation}\label{eq:Defue}
  \ue = \frac{\tilde{\we}}{\eps^{N+\frac12}}\eqqcolon \we \in L^2(0,T_0; H^1(\Omega)^2)\cap L^\infty(0,T_0;L^2_\sigma(\Omega)^2).
\end{equation}
The determination of $\ue$ involves solving a non-linear and non-local evolution equation. The equation's right-hand side contains terms depending on both $c_\mathrm{A}$ and $\ue$, in suitable function spaces. The proof can be done analogously as in \cite[Proof of Lemma~4.2]{StokesAllenCahn}.
By this choice we obtain
\begin{alignat}{2}\label{eq:cA}
  \partial_t c_\mathrm{A}+ \ve_\mathrm{A} \cdot \nabla c_\mathrm{A} + \eps^{N+\frac12}\zg\we|_{\Gamma}\cdot \nabla c_\mathrm{A} & = \Delta c_\mathrm{A}-\frac{f'(c_\mathrm{A})}{\eps^2} +\frac{1}{\varepsilon}\lambda_\mathrm{A}+S_\eps
\end{alignat}
in $\Omega\times (0,T_0)$.
For the following we assume that the estimates
\begin{subequations}\label{assumptions}
  \begin{align}
    \sup_{0\leq t\leq \tau} \|c_\eps(t) -c_\mathrm{A}(t)\|_{L^2(\Omega)} + \|\nabla (c_\eps -c_\mathrm{A}))\|_{L^2(\Omega\times(0,\tau)\setminus \Gamma(\delta))}                         & \leq R\eps^{\order+\frac12}, \\
    \|\nabla_\btau(c_\eps -c_\mathrm{A})\|_{L^2(\Omega\times(0,T_\eps)\cap \Gamma(2\delta))} +\eps\|\partial_\no(c_\eps -c_\mathrm{A})\|_{L^2(\Omega\times(0,\tau)\cap \Gamma(2\delta))} & \leq R\eps^{\order+\frac12}, \\
    \|\nabla(c_\eps -c_\mathrm{A})\|_{L^\infty(0,\tau;L^2(\Omega))}+\|\nabla^2(c_\eps -c_\mathrm{A})\|_{L^2(\Omega\times(0,\tau))}                                                       & \leq R\eps^{\order-\frac32}, \\
    \int_0^\tau\int_{\Gamma_t(\delta)}|\nabla u|^2 +\eps^{-2} f''(c_\mathrm{A}(x,t)) u^2 \dif x \dif t                                                                                   & \leq R^2\eps^{2N+1}.
  \end{align}
\end{subequations}
hold true for some $\tau=\tau(\eps) \in (0,T_0]$, $\eps_0\in (0,1]$, and all $\eps \in (0,\eps_0]$. Here $R=R(\theta)>0$ is chosen such that
\begin{equation}\label{initial assumption-0}
  \|c_{0,\eps}-c_{\mathrm{A},0}\|_{L^2(\Omega)}^2+ \varepsilon^4\|\nabla(c_{0,\eps}-c_{\mathrm{A},0})\|_{L^2(\Omega)}^2+\|\ve_{0,\eps}-\ve_{\mathrm{A},0}\|_{L^2(\Omega)}^2\leq \theta^2 \frac{R^2}{4}\eps^{2\order+1}e^{-C_LT_0}
\end{equation}
for all $\eps\in (0,1]$, where $C_L>0$ is the constant from Theorem~\ref{thm:Spectral} and $\theta \in (0,1]$ is a suitable constant to be determined later.
Finally let
\begin{align}
  T_{\varepsilon}\coloneqq \sup\{\tau\in[0,T_0]: \eqref{assumptions}  \ \text{holds true}.\}>0.\label{def:Teps}
\end{align}
The positivity of $T_\varepsilon$ follows from the initial estimate \eqref{initial assumption-0}, and the smoothness of $c_\varepsilon$ and $c_\mathrm{A}$, the continuity of $\|c_\varepsilon(t)-c_\mathrm{A}(t)\|_{\mathrm{L}^2(\Omega)}$ in time $t\in[0,T_0]$. It remains to show that $T_\varepsilon$ has a positive lower
bound that is independent of $\varepsilon$.

The following two theorems are taken from \cite[Theorems 4.1 and 4.2]{AbeFei23}. The proofs are given there.
\begin{theorem}\label{thm:weEstim}
  Assume that $c_\mathrm{A}, \tilde{\we}$ are as above and $\ue$ satisfies \eqref{assumptions} for some $R>0$, $\tau=T(\eps)$, and $N\geq 3$. Then there are some $C(R)>0$, $\eps_0>0$, and $M>0$ independent of $\eps \in (0,\eps_0)$ and $T\in (0,T_\varepsilon]$ such that $\|h_{N+\frac12}\|_{X_T}\leq M$ and
  \begin{equation}
    \|\we\|_{L^\infty(0,T;L^2(\Omega))}+\|\we\|_{L^2(0,T;H^1(\Omega))}\leq C(R).\label{neww1estimate}
  \end{equation}
  where $h_{N+\frac12}$ is as in Theorem~\ref{thm: h_N+1/2} with $\ue$ as in \eqref{eq:Defue} and $\we=\frac{\tilde{\we}}{\varepsilon^{N+\frac{1}{2}}}$.
\end{theorem}

\begin{theorem} For $T\in(0,T_\varepsilon)$ there holds
  \begin{alignat}{2}
    \int_0^{T}\bigg|\int_{\Omega}\big(\we-\we|_{\Gamma}\big)\cdot \nabla c_\mathrm{A}u\dif{x}\,\bigg|\dif t & \leq C(R)\varepsilon^{N+1}.\label{est:diffw1couple}
  \end{alignat}
\end{theorem}

\subsection{Proof of Theorem \ref{thm:main}}

In order to estimate the error due to linearization of $f'(c_\eps)$ we use from \cite[Proposition 4.3]{AbeFei23}:
\begin{prop} Let the same assumptions hold as in the beginning of Section \ref{subsec:prep}. Then
  \begin{align}
    \int_0^{T_\varepsilon}\int_{\Omega} |u|^3\dif x\,\dif t\leq C(R)T_\varepsilon^{\frac{1}{2}}\varepsilon^{3N+1}.\label{est:nonlinear}
  \end{align}
\end{prop}

\paragraph{Proof of Theorem \ref{thm:main}:} The central elements of the primary argument are as follows: It is sufficient to demonstrate that the function $T_\varepsilon$ possesses a positive lower bound $T_1\in (0,T_0]$ that is independent of the parameter $\varepsilon>0$ if $\varepsilon\leq \varepsilon_0$ for $\varepsilon_0>0$ sufficiently small. This will be accomplished by subtracting the approximate equations from the exact ones, employing the spectral estimate and the Gronwall inequality, together with the assumptions set forth in Subsection \ref{subsec:prep}. The next part of the argument will proceed by means of standard estimates, employing Hölder's and Young's inequalities. The crucial distinction between this approach and the ones in \cite[Section 4.2]{AbeFei23} is the inclusion of the mass conservation term. 
We note, that  \begin{equation}\label{velocity decompose new}
  \ve_\eps= \ve_\mathrm{A}+ \varepsilon^{N+\frac{1}{2}}\we
\end{equation}
with the help of \eqref{eq:NSAC3} and \eqref{eq:cA} we then find  that
\begin{align}\nonumber
  \partial_t u & +\ve_\eps\cdot \nabla u+ \eps^{N+\frac12}\big(\we-\we|_{\Gamma}\big)\cdot \nabla c_\mathrm{A} \\\label{eq:u2}
  &= \Delta u-\frac{f'(c_\eps)}{\eps^2}+\frac{f'(c_\mathrm{A})}{\eps^2} +{\varepsilon}^{-1}(\lambda_\varepsilon-\lambda_\mathrm{A})- S_\eps
\end{align}

First we estimate the mean value of $u$. To this end we take the mean value over \eqref{eq:u2} with respect to $x\in\Omega$ and obtain
\begin{align*}
  \frac{\dif}{\dif t} \fint_\Omega u\,dx + \eps^{N+\frac12}\fint_\Omega \big(\we-\we|_{\Gamma}\big)\cdot \nabla c_\mathrm{A}\,dx = \fint_\Omega \frac{f'(c_\mathrm{A})}{\eps^2} \,dx -\frac{\lambda_\mathrm{A}}\eps + O(\eps^{N+1})
\end{align*}
in $L^2(0,T_\eps)$ since $\lambda_\eps = \frac1\eps \int_\Omega f'(c_\eps)\,dx$ and due to \eqref{eq:SepsEstim1}. Moreover,
\begin{align*}
  &\eps^{N+\frac12}\int_\Omega |\big(\we-\we|_{\Gamma}\big)\cdot \nabla c_\mathrm{A}|\,dx\\
  &\leq C\eps^{N-\frac12} \int_{-2\delta}^{2\delta}\int_{\T^1} \underbrace{|(\tilde\we(r,s,t)-\tilde\we(0,s,t))|}_{\leq C\sqrt r \|\partial_r \tilde w(.,s)\|_{L^2((-2\delta,2\delta))}} \theta_0'(\rho_\eps)\,ds\, dr + C\eps^{N+1}\|\we(.,t)\|_{H^1(\Omega)}\\
  &\leq  C\eps^{N+1}\|\we(.,t) \|_{H^1(\Omega)}.
\end{align*}
Hence
\begin{equation*}
  \fint_\Omega \frac{f'(c_\mathrm{A})}{\eps^2} \,dx -\frac{\lambda_\mathrm{A}}\eps = O(\eps^{N+1})\qquad \text{in }L^2(0,T_\eps) 
\end{equation*}
and we can reformulate \eqref{eq:u2} as
\begin{align}\nonumber
  \partial_t u & +\ve_\eps\cdot \nabla u+ \eps^{N+\frac12}\big(\we-\we|_{\Gamma}\big)\cdot \nabla c_\mathrm{A}= \Delta u-P_0\left(\frac{f'(c_\eps)}{\eps^2}-\frac{f'(c_\mathrm{A})}{\eps^2}\right) - S'_\eps
  \\\label{eq:u}
               &= \Delta u-\frac{f''(c_\mathrm{A})}{\eps^2}u + \fint_{\Omega}\frac{f''(c_\mathrm{A})}{\eps^2}u\dif{x} -\eps^{-2}P_0\mathcal{N}(c_\mathrm{A},u)- S_\eps',
\end{align}
where $S_\eps'= O(\eps^{N+1})$ in $L^2(0,T_\eps,X_\eps')$, 
\begin{equation*}
  P_0 f= f-\fint_\Omega f(x)\,dx \qquad \text{for all }f\in L^1(\Omega)
\end{equation*}
and $\mathcal  {N}(c_\mathrm{A},u)=f'(c_\eps)-f'(c_\mathrm{A})-f''(c_\mathrm{A})u$. Because of the growth assumption $|f'''(s)|\leq C(|s|+1)$ for all $s\in\R$ and some $C>0$, we have by a Taylor expansion
\begin{equation*}
  |\mathcal  {N}(c_\mathrm{A},u)(x,t)|\leq C|u(x,t)|^2 \qquad \text{for all }x\in\Omega, t\in [0,T_0].
\end{equation*}
\noindent
\textit{Ad \eqref{eq:Error1}}. We test \eqref{eq:u} with $u$, integrate by parts and obtain:
  \begin{align*}
     & \frac{1}{2} \frac{\mathrm{~d}}{\mathrm{~d} t} \int_{\Omega} u^2 \mathrm{~d} x+\int_{\Omega}\left(|\nabla u|^2+\frac{f^{\prime \prime}\left(c_{\mathrm{A}}\right)}{\varepsilon^2} u^2\right) \mathrm{d} x  = \int_\Omega \frac{f^{\prime \prime}\left(c_{\mathrm{A}}\right)}{\varepsilon^2} u \dif x \fint_\Omega u(x)\dif x 
    \\&-\frac1{\varepsilon^2} \int_{\Omega}u P_0 \mathcal{N}\left(c_{\mathrm{A}}, u\right) \mathrm{~d} x-\int_{\Omega} S'_{\varepsilon} u \mathrm{~d} x
    - \varepsilon^{N+\frac{1}{2}}\int_\Omega (\mathbf{w}-\mathbf{w}|_\Gamma)\cdot\nabla c_\mathrm{A} u \dif{x}.
  \end{align*}
 By the definition of the $X_\eps$-norm, there is a decomposition $\tilde{u}(r,s,t)=Z(s,t)\eps^{-\frac12}\theta_0(\rho)+ \tilde{v}(r,s)$ in $\Gamma_t(\delta)$, where
  \begin{equation*}
    \|Z\|_{L^2(\T^1\times (0,T_\eps))}+ \|v\|_{L^2(0,T_\eps;H^1(\Gamma_t(\delta))}+\tfrac1\eps\|v\|_{L^2(0,T_\eps;L^2(\Gamma_t))}\leq CR\eps^{N+\frac12}.
  \end{equation*}
  This leads to
  \begin{align*}
   & \left|\int_\Omega \frac{f^{\prime \prime}\left(c_{\mathrm{A}}\right)}{\varepsilon^2} u \dif x \fint_\Omega u(x)\,dx\right| 
                                                                                                             \leq C\eps^{N-\frac12}\left|\int_{\Gamma_t(2\delta)} \zg f^{\prime \prime}\left(c_{\mathrm{A}}\right) u \dif x\right|  + O(\eps^{2N+1})\\
        &\quad \leq C\eps^{N-\frac12} \left|\int_{\Gamma_t(2\delta)} \zg f^{\prime \prime}\left(\theta_0(\rho_\eps(x,t))\right) \theta_0'(\rho_\eps(x,t))Z(S(x,t)) \dif x\right| + O(\eps^{2N+1})\\
     &\quad= C\eps^{N-\frac12}\left|\int_{\Gamma_t(2\delta)} \zg f^{\prime}\left(\theta_0(\rho_\eps(x,t))\right) Z(S(x,t)) \dif x\right| + O(\eps^{2N+1})\\
&\quad \leq O(\eps^{2N+1})
  \end{align*} 
  in $L^2(0,T_\eps)$ since $\left|\fint_\Omega u\dif x\right|=\left|\fint_\Omega c_A\dif x\right|\leq C\eps^{N+\frac32}$ uniformly in $t\in (0,T_\eps)$, 
  $\|u\|_{L^2(\Omega\times (0,T_\eps)\setminus \Gamma(\delta))}\leq C\eps^{N+\frac32}$ and Lemma~\ref{lem:MeanEstim}. Moreover, we have
  \begin{align*}
    -\frac1{\varepsilon^2} \int_{\Omega}u P_0 \mathcal{N}\left(c_{\mathrm{A}}, u\right)\, \mathrm{d} x &= - \frac1{\varepsilon^2} \int_{\Omega}u \mathcal{N}\left(c_{\mathrm{A}}, u\right) \, \mathrm{d} x + \frac1{\varepsilon^2}\int_\Omega \mathcal{N}\left(c_{\mathrm{A}}, u\right)\dif x \fint_\Omega u\dif x \\
    &\leq \frac{C}{\eps^2}\|u\|_{L^3(\Omega)}^3 + CR^2\eps^{3N+\frac12}
  \end{align*}
  uniformly in $t\in [0,T_\eps]$. Altogether we obtain
\begin{alignat}{2}\label{ineq:u energy-00}
   & \frac{1}{2}\frac{\dif}{\dif t}\int_{\Omega} u^2\dif x\,+ \int_{\Omega}\left( |\nabla u|^2+\frac{f''(c_\mathrm{A})}{\eps^2}u^2\right) \dif{x} \\\nonumber&\leq\eps^{N+\frac12}\bigg|\int_{\Omega}\big(\we-\we|_{\Gamma}\big)\cdot \nabla c_\mathrm{A}u\dif x\,\bigg|+\frac{C}{\eps^2}\int_{\Omega}|u|^3\dif x\,+\bigg|\int_{\Omega} S_\eps' u\dif x\,\bigg|+ CR^2\eps^{3N+\frac12}+O(\eps^{2N+1})
\end{alignat}
in $L^2(0,T_\eps)$. 
Applying the spectral estimate of Theorem \ref{thm:Spectral} in \eqref{ineq:u energy-00}, Gronwall's inequality, and \eqref{initial assumption-0}, one derives similarly as in \cite[Section 4.2]{AbeFei23}
\begin{alignat}{2}\label{ineq:u energy-1}
   & \sup_{0\leq t\leq T}\frac{1}{2}\int_{\Omega} u^2\dif x\, + \int_{0}^T\int_{\Om\setminus \Gamma_t(\delta)} |\nabla u|^2\dif x\, \dif t +  \int_{0}^T\int_{\Gamma_t({\delta})} |\nabla_\btau u|^2\dif{x} \dif t\nonumber \\&\leq e^{C_LT_0}\bigg(\frac{1}{2}\int_{\Omega} u^2|_{t=0}\dif x\, +\eps^{N+\frac12}\int_{0}^T\bigg|\int_{\Omega}\big(\we-\we|_{\Gamma}\big)\cdot \nabla c_\mathrm{A}u\,\dif x\,\bigg|\dif t\nonumber\\&\qquad+\eps^{-2}\int_{0}^T\int_{\Omega}|u|^3\dif x\,\dif t+\int_{0}^T\bigg|\int_{\Omega} S_\eps' u\dif{x}\,\bigg|\dif t+CR^2\eps^{3N+\frac12}+C\sqrt{T} \eps^ {2N+1}\bigg)
  \nonumber                                                                                                                                                                                                             \\&\leq \theta^2\frac{R^2}{8}\eps^{2N+1}+C(R,T_0)\big(\eps^{2N+\frac32}+\sqrt{T}\eps^{2N+1}\big)
\end{alignat}
since $N\geq 3$. Therefore, if $\varepsilon\in(0,\varepsilon_0)$ and $T\in (0,T_1]$ for $\varepsilon_0>0$ and $T_1\in (0,T_0]$ so small that
\begin{align}
  \theta^2\frac{R^2}{8}\eps^{2N+1}+C(R,T_0)\big(\eps^{2N+\frac32}+\eps^{3N-1}+\sqrt{T_1}\eps^{2N+1}\big)\leq \theta^2\frac{R^2}{4}\eps^{2N+1},\label{ineq:u energy-111}
\end{align}
we conclude
\begin{alignat}{2}\nonumber
  \sup_{0\leq t\leq T}\frac{1}{2}\int_{\Omega} u^2(x,t)\dif x\, &+ \int_{0}^T\int_{\Om\setminus \Gamma_t(\delta)} |\nabla u|^2\dif x\, \dif t\\\label{ineq:u energy-2}
  &+  \int_{0}^T\int_{\Gamma_t({\delta})} |\nabla_\btau u|^2\dif x \dif t\leq \theta^2\frac{R^2}{4}\eps^{2N+1}.
\end{alignat}
\textit{Ad \eqref{eq:Error2}.} The estimate \eqref{ineq:u energy-00} implies for $0\leq T\leq T_\varepsilon$:
\begin{alignat}{2}\label{ineq:u energy-3}
  & \int_0^T \int_{\Omega}\big( |\nabla u|^2+\frac{f''(c_\mathrm{A})}{\eps^2}u^2\big) \dif{x} \dif t\nonumber \\
  &\quad \leq\frac{1}{2}\int_{\Omega} u^2|_{t=0}\dif x\, +\eps^{N+\frac12}\int_{0}^T\bigg|\int_{\Omega}\big(\we-\we|_{\Gamma}\big)\cdot \nabla c_\mathrm{A}u\dif x\,\bigg|\dif t\nonumber\\&\qquad\qquad+\eps^{-2}\int_{0}^T\int_{\Omega}|u|^3\dif x\,\dif t+\int_{0}^T\bigg|\int_{\Omega}S_\eps u\dif x\,\bigg|\dif t+CR^2\eps^{3N+\frac12} +C\sqrt{T} \eps^{2N+1}
  \nonumber                                                                                                  \\
  &\quad \leq\theta^2\frac{R^2}{8}\eps^{2N+1}+C(R,T_0)(\eps^{2N+\frac32}+\sqrt{T} \eps^{2N+1}).
\end{alignat}
Hence for sufficienlty small  $\varepsilon_0>0$ and $T_1\in (0,T_0]$ and, if we choose $\theta\in (0,1]$ so small that $C_0 T_0 \theta^2\leq 1$,  where $C_0= -\min_{s\in\R} f''(s)$,  we have
\begin{alignat}{2}\label{ineq:u energy-4}
  \int_0^T \int_{\Omega}\big( |\nabla u|^2+\frac{f''(c_\mathrm{A})}{\eps^2}u^2\big) \dif{x} \dif t\leq\theta^2\frac{R^2}{4}\eps^{2N+1}
\end{alignat}
and
\begin{alignat}{2}\label{ineq:u energy-5}
  \eps^2\int_0^T \int_{\Omega}|\partial_{\nn} u|^2\dif{x} \dif t & \leq\eps^2\int_0^T \int_{\Omega}|\nabla u|^2\dif{x} \dif t
  \nonumber                                                                                                               \\&\leq\int_0^T \int_{\Omega}\big( \eps^2|\nabla u|^2+f''(c_\mathrm{A})u^2\big) \dif{x} \dif t+C_0\int_0^T \int_{\Omega}u^2\dif{x} \dif t\nonumber\\&\leq\frac{R^2}{4}\eps^{2N+3}+C_0T_0\theta^2\frac{R^2}{2}\eps^{2N+1}
  \leq\frac{3R^2}{4}\eps^{2N+1}
\end{alignat}
if $T\leq \min (T_1,T_\varepsilon)$ and $\varepsilon\in(0,\varepsilon_0)$.

\medskip

\noindent
\textit{Ad \eqref{eq:Error3}. } Testing \eqref{eq:u} by $-\varepsilon^4\Delta u$ one obtains as in \cite[Section 4.2]{AbeFei23} 
\begin{alignat}{2}\label{ineq:u energy-17}
  \sup_{0\leq t\leq T}\frac{1}{2}\int_{\Omega} |\nabla u|^2(x,t)\dif x\, + \int_{0}^T\int_{\Omega} |\Delta u|^2\dif x\, \dif t \leq \frac{2R^2}{3}\eps^{2N-3}
\end{alignat}
for sufficiently small $\eps_0\in (0,1]$.
We note that since $\fint_\Omega \Delta u\dif x=0$ and
\begin{equation*}
  \int_{\Omega}(P_0\mathcal{N}(c_\mathrm{A},u))\Delta u\dif x =   \int_{\Omega}\mathcal{N}(c_\mathrm{A},u)P_0\Delta u\dif x =   \int_{\Omega}\mathcal{N}(c_\mathrm{A},u)\Delta u\dif x 
\end{equation*}
the proof remains the same.

Because of the maximality of $T_\varepsilon$ in \eqref{def:Teps}, the estimates \eqref{ineq:u energy-2}, \eqref{ineq:u energy-4}, \eqref{ineq:u energy-5} and \eqref{ineq:u energy-17} imply  $T_\varepsilon\geq T_1$ and then \eqref{assumptions'} holds.
Finally,  \eqref{eq:convVelocityb} is a direct consequence of \eqref{velocity decompose} and \eqref{neww1estimate}. The remaining two statements in Theorem \ref{thm:main} follow from the constructions of $c_\mathrm{A}$ and  $\ve_\mathrm{A}$. Hence the proof of Theorem \ref{thm:main} is completed.

\appendix

\section{Formally Matched Asymptotics} \label{Appendix}
We will discuss the construction of approximate solutions using the method of formally matched asymptotics. There are two expansions, namely, the inner expansion in the tubular neighborhood of $\Gamma_t$ and the outer expansion outside away from the neighborhood. The scheme is similar to that in \cite{AlikakosLimitCH} with adaptations similar to those in \cite{ChenHilhorstLogak} and
\cite{StokesAllenCahn}. It is an adjustment as presented in \cite{AbelsMarquardt2} for the ``integer order part'' to the case of a Navier-Stokes/Allen-Cahn system instead of a Stokes/Cahn-Hilliard system, which can also be found in more detail in \cite{PhDMarquardt}.
First of all, we note that, since $
  \Div(\nabla c_{\eps}\otimes\nabla c_{\eps})=\frac{1}{2}\nabla\big(|\nabla c_{\eps}|^2\big)+\Delta c_{\eps}\nabla c_{\eps}$,
we  can rewrite \eqref{eq:NSAC1}-\eqref{eq:NSAC3} as follows
\begin{alignat}{2}\label{App eq:NSAC1-new}
  \partial_t \ve^\eps +\ve^\eps\cdot \nabla \ve^\eps-\Div(2\nu(c^\eps)D\ve^\eps)  +\nabla p^\eps & = -\eps\Delta c_{\eps}\nabla c_{\eps},                    \\\label{App eq:NSAC2-new}
  \Div \ve^\eps                                                                                  & = 0,                                                      \\ \label{App eq:NSAC3-new}
  \partial_t c^\eps +\ve^\eps\cdot \nabla c^\eps                                                 & =\Delta c_{\eps}-\e^{-2}f'(c_{\eps})+\e^{-1}\lambda_\eps,
\end{alignat}
in $\Omega\times (0,T_0)$ by redefining the pressure $p^\eps$. We will only focus on the formal asymptotics of the Allen-Cahn equation since the expansion of the Navier-Stokes part is the same as in \cite{AbeFei23}. For the sake of clarity, the ansatzes of the unknowns in the Allen Cahn equations are listed as follows.
The outer expansion is given by:
\begin{equation}\label{App eq: AnsatzOut}
  c_\mathrm{A}^\pm(x,t)  =\sum_{k=0}^{N+2}\eps^{k}c_{k}^{\pm}(x,t), \  \ \mathbf{v}_\mathrm{A}^\pm(x,t) =\sum_{k=0}^{N+2}\eps^{k}\mathbf{v}_{k}^{\pm}(x,t),\ \  p_\mathrm{A}^\pm(x,t) = \sum_{k=-1}^{N+2}\eps^{k}p_{k}^{\pm}(x,t)
\end{equation}
and the mass term expansion by \begin{equation} \label{App eq: Ansatz Mass}
  \lambda_{\mathrm{A}}(t)=\sum_{k=0}^{N+1} \varepsilon^k \lambda_k(t).
\end{equation}

\subsection{Outer Expansion in $\Omega^\pm$}\label{subsec: outer exp}
We insert the above ansatzes \eqref{App eq: AnsatzOut} and \eqref{App eq: Ansatz Mass} into \eqref{App eq:NSAC3-new}.
For this, we need the Taylor expansion of $f^\prime$, which is given by
\begin{equation}
  \begin{aligned}\label{eq: taylor f}
    f^{\prime}\left(c_\varepsilon\right) & =f^{\prime}\left(c_{0}^{ \pm}\right)+\varepsilon f^{\prime \prime}\left(c_{0}^{ \pm}\right) c_{1}^{ \pm}+\sum_{k=2}^{N + 2} \varepsilon^{k}\left(f^{\prime \prime}\left(c_{0}^{ \pm}\right) c_{k}^{ \pm}+f_{k-1}\left(c_{0}^{ \pm}, \ldots, c_{k-1}^{ \pm}\right)\right) \\
                                         & +\varepsilon^{N+3} f_{N+3}^{\varepsilon}\left(c_{0}^{ \pm}, \ldots, c_{N+2}^{ \pm}\right)\eqqcolon\sum_{k\geq0} \varepsilon^{k} f_{k}^{ \pm},
  \end{aligned}
\end{equation}
where the functions $f_{k}$ are polynomials in $\left(c_{1}^{ \pm}, \ldots, c_{k}^{ \pm}\right)$. Thus we get
\begin{equation}
  \begin{aligned}
     & \sum_{k\geq0} \varepsilon^{k} \partial_{t} c_{k}^{ \pm}
    + \sum_{k\geq 0}\varepsilon^k \sum_{0\leq j \leq k} \mathbf{v}_j^\pm \cdot \nabla c_{k-j}^\pm  =-\frac{1}{\varepsilon^{2}} f^{\prime}\left(c_{0}^{ \pm}\right)-\frac{1}{\varepsilon} f^{\prime \prime}\left(c_{0}^{ \pm}\right) c_{1}^{ \pm} \\
     & +\sum_{k\geq0} \varepsilon^{k}\left(\Delta c_{k}^{ \pm}-f^{\prime \prime}\left(c_{0}^{ \pm}\right) c_{k+2}^{ \pm}-f_{k+1}\left(c_{0}^{ \pm}, \ldots, c_{k-1}^{ \pm}\right)\right)
    +\sum_{k\geq0} \varepsilon^{k-1} \lambda_{k}.
  \end{aligned}
\end{equation}
By matching the coefficients, we get
for all $\varepsilon^{k-2}$ terms with $k\geq 2$:
\begin{equation}
  \begin{aligned}
     c_{k}^{ \pm}=\frac{1}{{f^{\prime\prime}}(\pm 1)}\left[-\partial_{t} c_{k-2}^{ \pm}-f_{k-1}\left( \pm 1, c_{1}^{ \pm},\dots, c_{k-1}^{ \pm}\right)+\lambda_{k-1}\right],\,c_{0}^{ \pm}= \pm 1,\,c_{1}^{ \pm}=\frac{\lambda_{0}}{f^{\prime \prime}( \pm 1)}, \label{eq: coeff outer eq}
  \end{aligned}
\end{equation}
for all $t\in (0,T_0)$ and independent of $x$.

\subsection{Inner Expansion in $\Gamma(2\delta)$}

Close to the interface $\Gamma$, we introduced the {stretched variable}:
$$
  \rho_{\varepsilon}(x, t)\coloneqq\frac{d_{\Gamma}(x, t)-\varepsilon h_{\varepsilon}(S(x, t), t)}{\varepsilon} \quad \text { for }(x, t) \in \Gamma(2\delta), \, \varepsilon \in(0,1).
$$
Here $h_{\varepsilon}\colon \T^{1} \times[0, T_0] \rightarrow \mathbb{R}$ is a given smooth function and can heuristically be interpreted as the distance of the zero level set of $c_{\varepsilon}$ to $\Gamma$.
In the following, we write $\rho\coloneqq\rho_{\varepsilon}(x, t)$.
When referring to $\tilde{c}$ and the expansion terms, we write $\nabla=\nabla_x$ and $\Delta=\Delta_x$, for the gradient and Laplacian, resp., with respect to $x$.
The derivatives $\partial_t^{\Gamma} h_{\varepsilon}(x, t)$, $\nabla^{\Gamma} h_{\varepsilon}(x, t)$, $\Delta^{\Gamma} h_{\varepsilon}(x, t)$, $\partial_{t}^2 h_{\varepsilon}(x, t)$ are for $(x, t) \in \Gamma(2 \delta)$ to be understood in the sense of \eqref{Prelim:1.13}-\eqref{Prelim:1.12}.
In $\Gamma(2\delta)$ use the Ansatz
\begin{equation}
  \tc_{\mathrm{A}}^\mathrm{in}(x, t)=\tilde{c}_{\mathrm{A}}^\mathrm{in}\left(\frac{d_\Gamma(x, t)}{\varepsilon}-h_{\varepsilon}(S(x, t), t), x, t\right)
\end{equation}
for some $\tilde{c}_\mathrm{A}^\mathrm{in}: \mathbb{R} \times \Gamma(2\delta) \rightarrow \mathbb{R}$ and analogously $\tv_\mathrm{A}^{\mathrm{in}}(x,t)=\tilde{\ve}_\mathrm{A}^{\mathrm{in}}(\rho, x, t)$, $\tp_\mathrm{A}^{\mathrm{in}}(x,t)=\tilde{p}_\mathrm{A}^{\mathrm{in}}(\rho, x, t)$.
Furthermore, we assume that we have the inner expansions
\begin{align}
  \tilde{c}_\mathrm{A}^\mathrm{in}(\rho, x, t) & = \sum_{k=0}^{N+2} \varepsilon^k c_k(\rho, x, t) &  & \text{for all $(\rho, x, t) \in \mathbb{R} \times \Gamma(2 \delta)\times [0,T_0]$}, \label{eq:c in} \\
  h_{\varepsilon}(s, t)                        & = \sum_{k=0}^{N} \varepsilon^k h_{k+1}(s, t)     &  & \text {for all } s \in \T^{1},\, t \in[0, T_0] \label{eq: h as fin power ser},                      \\
  \tilde{\ve}_\mathrm{A}^{\mathrm{in}}(\rho, x, t)     & = \sum_{k=0}^{N+2}\eps^k \mathbf{v}_k(\rho,x,t)  &  & \text{for all $(\rho, x, t) \in \mathbb{R} \times \Gamma(2 \delta)\times [0,T_0]$},                 \\
  \tilde{p}_\mathrm{A}^{\mathrm{in}}(\rho, x, t)     & =\sum_{k=-1}^{N+1}\eps^k p_k(\rho,x,t)           &  & \text{for all $(\rho, x, t) \in \mathbb{R} \times \Gamma(2 \delta)\times [0,T_0]$}\label{eq:p in}
\end{align}
where $N\in\N$, $c_k\colon\mathbb{R} \times \Gamma(2 \delta) \rightarrow \mathbb{R}$ and $h_{k+1}: \T^{1} \times[0, T_0] \rightarrow \mathbb{R}$ are smooth functions for $k \in\{0,\ldots, N\}$.
In order to join the inner and outer expansions, it is necessary to use the inner-outer matching conditions, which control all the derivatives of the coefficients, given by:

\begin{defn}[Inner-outer matching condition]~\\ Let $m,n,l\geq 0$. For all $k \geqslant 0$ and some constants $\alpha,C>0$ and all $\rho>0, $ we define the \emph{inner-outer matching condition} for $c_k$ as
  \begin{equation}
    \sup_{(x,t)\in\Gamma(2\delta)}|\partial_x^m \partial_t^n \partial_\rho^l\left[c_k\left( \pm \rho, x,t\right)-c_k^{ \pm}(x, t)\right]| \leqslant C e^{-\alpha \rho}\label{eq: matching cond c} .
  \end{equation}\end{defn}
In the new coordinates $(\rho, x, t)$, the Allen-Cahn equation is
\begin{align}
  \begin{split}
    \partial_\rho^2 \tilde{c}_{\varepsilon}-f^{\prime}\left(\tilde{c}_{\varepsilon}\right)= & \varepsilon\left(\partial_\rho \tilde{c}_{\varepsilon} \partial_t d_{\Gamma}+\partial_\rho \tilde{c}_{\varepsilon} {\mathbf{v}_\varepsilon} \cdot \nabla d_{\Gamma}-2 \nabla \partial_\rho \tilde{c}_{\varepsilon} \cdot \nabla d_{\Gamma}-\partial_\rho \tilde{c}_{\varepsilon} \Delta d_{\Gamma}+\lambda_\varepsilon\right) \\
                                                                                            & +\varepsilon^2\left(2 \nabla \partial_\rho \tilde{c}_{\varepsilon} \cdot \nabla^{\Gamma} h_{\varepsilon}+\partial_\rho \tilde{c}_{\varepsilon} \Delta^{\Gamma} h_{\varepsilon}-\partial_\rho^2 \tilde{c}_{\varepsilon}\left|\nabla^{\Gamma} h_{\varepsilon}\right|^2-\Delta \tilde{c}_{\varepsilon}\right.                    \\
                                                                                            & \left.+\partial_t \tilde{c}_{\varepsilon}-\partial_\rho \tilde{c}_{\varepsilon} \partial_t^{\Gamma} h_{\varepsilon}+{\mathbf{v}_\varepsilon} \cdot\left(\nabla \tilde{c}_{\varepsilon}-\partial_\rho \tilde{c}_{\varepsilon} \nabla^{\Gamma} h_{\varepsilon}\right)\right) , \label{eq: PDE in new coordinates}
  \end{split}
\end{align}
which should hold in $$S^\varepsilon = \left\{(\rho,x,t)\in\R\times\Gamma(2\delta)|\rho= \frac{d_\Gamma(x,t)}{\varepsilon}-h_\varepsilon(S(x,t),t)\right\}.$$

\begin{rem}[normalization of inner coefficients $c_k$]~\\
  We normalize the coefficients $c_k$ such that $c_k(0, x, t)=0$ for all $(x, t) \in \Gamma(2 \delta)$ and $ k \geqslant 0$.
\end{rem}

It seems reasonable at this point to conclude the discussion by inserting the Ansatz \eqref{eq:c in}-\eqref{eq:p in} into the PDE \eqref{eq: PDE in new coordinates} from before in a manner consistent with the traditional approach of matched asymptotic expansions. However, this would yield solutions that grow as polynomials of $\rho$ as $|\rho|\to\infty$. This was demonstrated in the reference \cite{DeMottoniSchatzman}. It follows that our matching conditions, as in equation \eqref{eq: matching cond c}, cannot be fulfilled. It is substantial to recall that equation \eqref{eq: PDE in new coordinates} must only hold in $S^\varepsilon$. Consequently, we can add a function that vanishes in $S^\varepsilon$. This function can be selected in such a way that it ensures the inner outer matching condition. The following lemma states the sufficient condition for \eqref{eq: PDE in new coordinates} to have a unique solution that satisfies a matching condition.

\begin{lem}[optimal profile, solvability condition I]\label{lem: solv cond}~\\ Let be $f$ as before.
  The ODE
  \begin{equation}
    \theta_{0}^{\prime \prime}-f'\left(\theta_{0}\right)=0 \text { on } \mathbb{R},\quad \theta_{0}( \pm \infty)= \pm 1,\quad \theta_{0}(0)=0\label{eq: optimal profile}
  \end{equation}
  has a unique solution $\theta_0$, the \emph{optimal profile} satisfying $\partial_\rho^k (\theta_0\pm 1) = {O}(\mathrm{e}^{-\alpha |\rho|})$with  $\alpha\coloneqq\min \sqrt{f^{\prime\prime}(\pm1)}$, $k\in\N_0$ and $\rho\to \pm \infty$. Let $\mathcal{L}$ be defined as
  $$
    \mathcal{L} v\coloneqq -v^{\prime \prime}+f^{\prime \prime}\left(\theta_{0}\right) v .
  $$
  Assume that a function $h(\rho, s, t)$ satisfies, as $\rho \rightarrow \pm \infty$,
  \begin{equation}
    \partial_{\rho}^{m} \partial_{s}^{n} \partial_{t}^{l}\left[h(\rho, s, t)-h^{ \pm}(t)\right]=O\left(\mathrm{e}^{-\alpha |\rho|}\right)
  \end{equation}
  for all $(m, n, l) \in \mathbb{N}_0^{3}$ and $(s, t)$ in $\T^{1}\times[0,T_0]$. Then
  $$
    \mathcal{L} v=h(\cdot, s, t)\quad \text{ in } \mathbb{R} \text{ with } v(0, s, t)=0
  $$
  has a unique bounded solution $v(\rho, s, t)$ if and only if for all $(s, t) \in\T^{1} \times[0,T_0]$ it holds
  \begin{equation}
    \int_{\mathbb{R}} h(\rho, s, t) \theta_{0}^{\prime}(\rho) \dif\rho=0 .\tag{\text{\textit{solvability condition}}}   \label{eq: solv cond}
  \end{equation}
  Moreover, if the solution exists, then it satisfies for all $(m, n, l) \in \mathbb{N}_0^{3}$ and $(s, t) \in \T^{1} \times[0, T_0]$:
  $$
    \partial_{l}^{m} \partial_{s}^{n} \partial_{t}^{l}\left[v(\rho, s, t)-\frac{h^{ \pm}(t)}{f^{\prime\prime}( \pm 1)}\right]={O}\left(\mathrm{e}^{-\alpha |\rho|}\right) \text{ as } \rho\rightarrow \pm \infty .
  $$
\end{lem}
\begin{proof} The part with the optimal profile $\theta_0$ follows from \cite[Lemma 2.6.1]{PromotionStefan} and the rest from \cite[Lemma 2.6.2]{PromotionStefan}.
\end{proof}
\begin{rem}\label{rem: theta odd}
  Since $f$ is even, $f'$ is odd. The ODE of $\theta_0$ and its conditions are invariant under reflection; by uniqueness this yields $\theta_0(-\rho)=-\theta_0(\rho)$, i.e. the latter is an odd function.
\end{rem}

We see that the above lemma not only gives us the existence of a solution to \eqref{eq: PDE in new coordinates}. After matching the coefficients in front of the term some order in $\eps$, it reduces the PDE into an ODE in the stretched variable $\rho$ by fixing $(x,t)$. Nevertheless, the \eqref{eq: solv cond} has to be fulfilled. For this reason, we add an auxiliary function $g_{\varepsilon}(x, t)$ for $(x, t) \in \Gamma(2 \delta)$. Thus, (\ref{eq: PDE in new coordinates}) becomes
\begin{equation}
  \begin{aligned}
    \partial_\rho^2 \tilde{c}_{\varepsilon}-f^{\prime}\left(\tilde{c}_{\varepsilon}\right)= & \varepsilon\left(\partial_\rho \tilde{c}_{\varepsilon} \partial_t d_{\Gamma}+\partial_\rho \tilde{c}_{\varepsilon} {\mathbf{v}_\varepsilon} \cdot \nabla d_{\Gamma}-2 \nabla \partial_\rho \tilde{c}_{\varepsilon} \cdot \nabla d_{\Gamma}-\partial_\rho \tilde{c}_{\varepsilon} \Delta d_{\Gamma}-\lambda_\varepsilon\right) \\
                                                                                            & +\varepsilon^2\left(2 \nabla \partial_\rho \tilde{c}_{\varepsilon} \cdot \nabla^{\Gamma} h_{\varepsilon}+\partial_\rho \tilde{c}_{\varepsilon} \Delta^{\Gamma} h_{\varepsilon}-\partial_\rho^2 \tilde{c}_{\varepsilon}\left|\nabla^{\Gamma} h_{\varepsilon}\right|^2-\Delta \tilde{c}_{\varepsilon}\right.                    \\
                                                                                            & \left.+\partial_t \tilde{c}_{\varepsilon}-\partial_\rho \tilde{c}_{\varepsilon} \partial_t^{\Gamma} h_{\varepsilon}+{\mathbf{v}_\varepsilon} \cdot\left(\nabla \tilde{c}_{\varepsilon}-\partial_\rho \tilde{c}_{\varepsilon} \nabla^{\Gamma} h_{\varepsilon}\right)\right)                                                    \\
                                                                                            & +\varepsilon g_{\varepsilon} \theta_0^{\prime}(\rho)\left(d_{\Gamma}-\varepsilon\left(\rho+h_{\varepsilon}\right)\right).\label{eq: PDE in new coord with aux fct}
  \end{aligned}
\end{equation}
It is evident that the added term in the last line is zero in $S^\varepsilon$. After matching the coefficients in front of powers $\eps^k$, we consider (\ref{eq: PDE in new coord with aux fct}) as ODE in $\rho \in \mathbb{R}$ where $(x, t) \in \Gamma(2 \delta)$ are seen as fixed parameters. This idea is inspired by \cite[pg. 179]{PromotionStefan}.
We assume that the auxiliary function has an expansion of the form
\begin{equation} \label{eq App: g}    g_{\varepsilon}(x, t) = \sum_{k=0}^{N+1} g_{k}(x, t) \varepsilon^{k}
  \qquad \text{for } (x, t) \in \Gamma(2 \delta), \eps \in (0,1] \text {. }
\end{equation}
The subsequent equations hold for $(x,t)\in\Gamma(2\delta)$ and for $\rho\in\R$. For the $\varepsilon^0$ terms, we have
\begin{equation}
  \partial_\rho^{2} c_{0}-f^{\prime}(c_{0})=0\label{eq: 0 order inner expansion c},
\end{equation}
which is solved by $ c_{0}=\theta_{0}$ with the optimal profile $\theta_0$ (cf.~Lemma \ref{lem: solv cond}). The $\varepsilon^1$-terms yield
\begin{equation*}%\label{eq: 1st order c}
  \partial_\rho^2 c_1-f^{\prime \prime}(\theta_0) c_1  =\theta_0^{\prime}(\rho)\left(\partial_t d_{\Gamma}+\mathbf{v}_0 \cdot \nabla d_{\Gamma}-\Delta d_{\Gamma}+ g_0 d_{\Gamma}\right)-\lambda_0.
\end{equation*}
Testing the right-hand side with $\theta_0'$, we obtain
$$
  \sigma (\partial_t d_\Gamma + \mathbf{v}_0 \cdot \nabla d_\Gamma- \Delta d_\Gamma+ g_0 d_{\Gamma})-2\lambda_0 = 0
$$
because of $\int_\mathbb{R}(\partial_\rho^2 c_1 + f''(\theta_0)c_1)\theta_0'\dif \rho = \int_\mathbb{R}c_1 \mathcal{L}\theta _0'\dif \rho = 0$. On $\Gamma$ this equation is equivalent to \eqref{eq:Limit5} (for a suitable time dependent constant $\lambda_0$). On $\Gamma(2\delta)\backslash \Gamma$ this equation is equivalent to
\begin{equation}
  g_0=-\frac1{d_\Gamma}\left(\partial_t d_\Gamma + \mathbf{v}_0 \cdot \nabla d_\Gamma- \Delta d_\Gamma -\frac{2}{\sigma }\lambda_0\right) \label{eq:g0}
\end{equation}
such that \eqref{eq: solv cond} is fulfilled. Here $g_0$ extends to a smooth function on $\Gamma(2\delta)$ since the latter bracket vanishes on $\Gamma$. Altogether we obtain
\begin{equation}\label{eq: 1st order c}
  \partial_\rho^2 c_1-f^{\prime \prime}(\theta_0) c_1   =-\left(1-\frac{2}{\sigma}\theta_0'(\rho)\right)\lambda_0,
\end{equation}
In particular, we see $c_1=c_1(\rho,t)$ independent of $x$. Moreover, $c_1$ is even with respect to $\rho$ since the right-hand side and $f''(\theta_0)$ are even with respect to $\rho\in\R$. 
\noindent For the $\varepsilon^2$ order, we have:\begin{align}
  \partial_{\rho}^{2} c_{2}-f^{\prime \prime}(\theta_{0}) c_{2} & =B^{1}-\lambda_{1}+\tfrac12f^{\prime\prime\prime}(\theta_0)c_1^2\label{eq: c2 ode},
\end{align}
where
\begin{align*}
  B^{1} & =\theta_{0}^{\prime}(\rho)\left(\Delta^{\Gamma} h_{1}-\ve_0\cdot \nabla^\Gamma h_1-\ve_1\cdot \nabla d_\Gamma-\partial_{t}^{\Gamma} h_{1}-g_{0} h_{1}+g_{1} d_{\Gamma}\right)-\theta_{0}^{\prime \prime}(\rho) |\nabla^{\Gamma} h_{1}|^2 \\
        & \quad +\tilde{B}^{1}(c_0,c_1,\ve_0).
\end{align*}
The remainder $\tilde{B}^{1}$ has an exponential decay as $|\rho| \rightarrow \infty$ by the inner outer matching condition and is given by
\begin{equation}
  \label{eq:Btilde 1}
  \tilde{B}^1= \partial_\rho c_1(\rho)(\partial_t d_\Gamma+\mathbf{v}_0 \cdot\nabla d_\Gamma+\Delta d_\Gamma)-2\rho\theta_0'(\rho)g_0.
\end{equation} Testing with $\theta_0'(\rho)$, yields the solvability condition:
\begin{equation} \label{eq: g1dGamma}
  \begin{aligned}
    g_1 d_\Gamma & = \partial_{t}^{\Gamma} h_{1}+\ve_0\cdot \nabla^\Gamma h_1+\ve_1\cdot \nabla d_\Gamma-\Delta^{\Gamma} h_{1}+g_{0} h_{1}+2\frac{\lambda_1}\sigma - \frac1\sigma\int_{\R} \left(\tilde{B}^1 +\tfrac12 f'''(\theta_0)c_1^2\right)\theta_0'\, d\rho \\
                 & =\partial_{t}^{\Gamma} h_{1}+\ve_0\cdot \nabla^\Gamma h_1+\ve_1\cdot \nabla d_\Gamma-\Delta^{\Gamma} h_{1}+g_{0} h_{1}+2\frac{\lambda_1}\sigma,
  \end{aligned}
\end{equation}
because the last integral vanishes due to the evenness of $c_1$.
On $\Gamma$ this determines the evolution equation for $h_1$ and on $\Gamma(2\delta)\setminus \Gamma$ this determines $g_1$. Then \eqref{eq: c2 ode} simplifies to:
\begin{equation}
  \label{eq: c2 ode simplified}
  \partial_\rho ^2c_2 -f''(\theta_0)c_2 = -\left(1-\frac{2}{\sigma}\theta_0'(\rho)\right)\lambda_1+\frac{1}{2}f'''(\theta_0(\rho))c_1^2(\rho)-\theta_0''(\rho)|\nabla^\Gamma h_1|^2+\tilde{B}^1
\end{equation}
In general, for $\varepsilon^k$ terms with $k\geq 3$:
\begin{align}
  \partial_{\rho}^{2} c_{k}-f^{\prime \prime}\left(\theta_{0}\right) c_{k} & =B^{k-1}-\lambda_{k-1}+f_{k-1}(c_0,\dots,c_{k-1})+\partial_{t}c_{k-2}\label{eq: coeff inner eq},
\end{align}
where
$$
  \begin{aligned}
    B^{k-1} & =\theta_{0}^{\prime}(\rho)\left(\Delta^{\Gamma} h_{k-1}-\partial_{t}^{\Gamma} h_{k-1}\right)-2 \theta_{0}^{\prime \prime}(\rho) \nabla^{\Gamma} h_{1} \cdot \nabla^{\Gamma} h_{k-1}-\theta_{0}^{\prime}(\rho) g_{0} h_{k-1} \\
            & +\theta_{0}^{\prime}(\rho) g_{k-1} d_{\Gamma}+\tilde{B}^{k-1}
  \end{aligned}
$$
and $\tilde{B}^{k-1}$ depend on $c_0,\dots,c_{k-1}$, $\mathbf{v}_0,\dots,\mathbf{v}_{k-2}$ and $h_1,\dots, h_{k-2}$ determined by an induction argument, which have an exponential decay as $|\rho| \rightarrow \infty$ by the inner outer matching condition.

\begin{rem}
  \label{rem: c1 even} Recall that the optimal profile $\theta_0$ is odd and thus $\theta_0'$ is even in $\rho$. Moreover, in \eqref{eq: 1st order c} the right-hand side does only depend on $\rho$ via $\theta_0'(\rho)$ and is therefore even. The linear differential operator $\mathcal{L}c_1 =  -\partial_\rho^2 c_1 + f''(\theta_0(\rho))\,c_1$ has even coefficients and hence preserves parity. Consequently, the unique bounded solution $\hat c_1(\rho,s,t)$ is even in $\rho$.
\end{rem}

\subsection{Existence of the Expansion Terms}
\subsubsection{Solving Zeroth Order Terms}

The most important contribution in the expansion, is the zeroth order term since in the sharp interface limit, we take $\varepsilon\to 0$. The only term left in the limit will be the zeroth order term. 
\begin{lem}[zeroth order terms] We define the terms of
  \begin{itemize}

  \item \textit{the outer expansion} as  $c_{0}^{ \pm}(x, t)= \pm 1$ for $(x, t) \in \Omega^{ \pm} \cup \Gamma_{T_0}(2 \delta)$.
    \item \textit{the inner expansion} as $c_{0}\left(\rho, x, t\right)=\theta_{0}(\rho)$ for all $\left(\rho,x, t\right) \in \mathbb{R} \times \Gamma(2 \delta)$.
  \end{itemize}
  Then $c_0^\pm$ and $c_0$ fulfill the outer and inner equation respectively in zeroth order. They are smooth and bounded and satisfy the inner-outer matching condition \eqref{eq: matching cond c} for $k=0$.
\end{lem}
\begin{proof}
  \textit{Ad (\ref{eq: matching cond c})}. The inner outer matching condition is fulfilled  since $\partial_{\rho}^{m}[\theta_0(\rho)\pm 1] = {O}(\mathrm{e}^{-\alpha \rho}) $ as $\rho \to \infty$ in $\mathrm{L}^\infty(\Gamma(2\delta))$, $m\in\N_0$.
\end{proof}

%%%%%%%%%%%%%%%%%%%%%%%%%%%%%%%%%%%%%%%%%%%%%%%%%%%%%%%%%%%%%%%%%%%%%%%%%%%%%%%%%%%%%%%

\subsubsection{Determination of the Higher Order Terms}
In this subsection, we calculate the exact form of the $g_i$ for $i\geq 1$ as in \eqref{eq App: g}. We have the freedom to choose $g_i$, such that \eqref{eq: solv cond} for \eqref{eq: coeff inner eq} is fulfilled. By construction, the terms with $g_i d_\Gamma$ in \eqref{eq: coeff inner eq} vanish on the interface $\Gamma$. Such functions can be extended in the normal direction of the interface $\Gamma$ as follows from \cite[pg. 428]{StokesAllenCahn}:
\begin{lem}[extension on the interface]\label{lem: extension interface} If $g\colon \Gamma(2\delta)\to \R$ is smooth in normal direction and vanishes on the interface $\Gamma$, then
  \begin{equation}
    \tilde{g}(x,t) = \begin{cases} \frac{g(x.t)}{d_\Gamma (x,t)} & \text{if }(x.t)\in\Gamma(2\delta)\backslash\Gamma, \\
              \partial_\mathbf{n} g(x,t)    & \text{if }(x,t)\in\Gamma
    \end{cases}
  \end{equation}
  is smooth in the normal direction. Moreover, $\tilde{g}$ is smooth if $g$ is also smooth.
\end{lem}
\iffalse
  \begin{proof} For the relation on the interface $\Gamma$, we use \eqref{eq:1.4} and Taylor expand the function around $r=0$:
    \begin{equation*}
      \frac{g(x,t)}{d_\Gamma} = \frac{\tilde{g}(r,s,t)-\tilde{g}(0,s,t)}{r} = \partial_r \tilde{g}(0,s,t)+{O}(r) \text{ for }r\to 0.
    \end{equation*}
    Since $\partial_r \tilde{g}(0,s,t) = \partial_\mathbf{n}{g}(P_{\Gamma_t}(x),t)$, we get the desired identity. The smoothness of $\tilde{g}$ follows from
    $$\begin{aligned}
        \int_0^1 \partial_r g(\theta r, s, t) \dif\theta & =\frac{g(r, s, t)}{r}-\frac{1}{r^2} \int_0^r g(u, s, t) \dif u                                                                                \\
                                                         & =\tilde{g}(r, s, t)-\frac{1}{r} \int_0^r \tilde{g}(u, s, t) \dif u                                                                            \\
                                                         & =\int_0^r\underbrace{ \frac{\tilde{g}(r, s, t)-\tilde{g}(u, s, t)}{r}}_{=\ \partial_r \tilde{g}( u, s, t)+{O}(r) \text{ for $r\to 0$}} \dif u \\
      \end{aligned}
    $$
    since the left hand side is smooth.
  \end{proof}\fi

\begin{lem}[higher order terms]\label{lem: higher order terms}
  Let $k \in\{1, \ldots, N+2\}$ and the coefficients $\lambda_{k-1}$ of the asymptotic expansion $\lambda_\mathrm{A}$ (cf.~Theorem \ref{lem: coeff lambda}) be given. Furthermore, the $h_{i-1}$, $i\in \{2,\ldots, N+2\}$, satisfies equation \eqref{eq: time der of h} as demonstrated in the following proof. Then there are unique smooth functions $c_{k}$ and $c_{k}^{ \pm}$ which are bounded on their respective domains such that, for the $k$-th order the outer equations \eqref{eq: coeff outer eq}, the inner equations \eqref{eq: coeff inner eq}, the inner-outer matching conditions \eqref{eq: matching cond c} are satisfied.
\end{lem}
\begin{proof} \textit{Ad inner coefficients}.
  For the first order, $k=1$, the solvability condition is checked in the last section.
  By applying Lemma \ref{lem: solv cond} to \eqref{eq: coeff inner eq}, we can get $c_{k}$ for the other $k\geq 2$. The solvability condition for \eqref{eq: coeff inner eq} reads
  \begin{equation}
    \begin{aligned}
      \partial_{t}^{\Gamma} h_{k-1}
      +\mathbf{v}_0 \cdot \nabla^{\Gamma} h_{k-1}-\Delta^{\Gamma} h_{k-1}+\sigma^{-1}\int_\mathbb{R} \nabla d_\Gamma \cdot \mathbf{v}_{k-1}(\theta_0'(\rho)^2) \mathrm{d} \rho \\
      +g_{0} h_{k-1}-g_{k-1} d_{\Gamma}+2\sigma^{-1} \lambda_{k-1}=\sigma^{-1}\int_{\mathbb{R}} \tilde{B}^{k-1} \theta_{0}^{\prime}(\rho) \dif \rho,
    \end{aligned}\label{eq: time der of h}
  \end{equation}
  using $\sigma\coloneqq{\int_{\R} (\theta^\prime_{0}(\rho))^2 \dif {\rho}}$  and $\int_\R\theta_0''(\rho)\theta_0'(\rho)\dif \rho=0$. We apply Lemma \ref{lem: extension interface}. To this end we have on $\Gamma\left(\right.$since $\left.g_{k-1}d_\Gamma=0\right)$:
  $$
    \begin{aligned}
      \partial_{t}^{\Gamma} h_{k-1}+\mathbf{v}_0 \cdot \nabla^{\Gamma} h_{k-1}-\Delta^{\Gamma} h_{k-1}+\sigma^{-1}\int_\mathbb{R} \nabla d_\Gamma \cdot \mathbf{v}_{k-1}(\theta_0'(\rho)^2) \mathrm{d} \rho \\
      +g_{0} h_{k-1}+2\sigma^{-1} \lambda_{k-1}=\sigma^{-1} \int_{\mathbb{R}} \tilde{B}^{k-1} \theta_{0}^{\prime}(\rho) \dif \rho.
    \end{aligned}
  $$
  We take in $\Gamma(2 \delta) \backslash \Gamma$:
  \begin{equation}
    \begin{aligned}
      g_{k-1}= \frac{1}{d_\Gamma}\left[
      \partial_{t}^{\Gamma} h_{k-1}+\mathbf{v}_0 \cdot \nabla^{\Gamma} h_{k-1}-\Delta^{\Gamma} h_{k-1}+\sigma^{-1}\int_\mathbb{R} \nabla d_\Gamma \cdot \mathbf{v}_{k-1}(\theta_0'(\rho)^2) \mathrm{d} \rho \right. \\
      \left.   +g_{0} h_{k-1}-\Delta^{\Gamma} h_{k-1}-{\sigma}^{-1}\int_{\mathbb{R}} \theta_{0}^{\prime}(\rho) \tilde{B}^{k-1} \dif \rho+{2\sigma^{-1}} \lambda_{k-1}\right]
    \end{aligned}
  \end{equation}
  and on $\Gamma$:
  \begin{equation}
    \begin{aligned}
      g_{k-1}=\mathbf{n} \cdot \nabla \left[\partial_{t}^{\Gamma} h_{k-1}+\mathbf{v}_0 \cdot \nabla^{\Gamma} h_{k-1}-\Delta^{\Gamma} h_{k-1}+\sigma^{-1}\int_\mathbb{R} \nabla d_\Gamma \cdot \mathbf{v}_{k-1}(\theta_0'(\rho)^2) \mathrm{d} \rho \right. \\
      \left. +g_{0} h_{k-1}-\Delta^{\Gamma} h_{k-1}-{\sigma}^{-1} \int_{\mathbb{R}} \theta_{0}^{\prime}(\rho) \tilde{B}^{k-1} \dif \rho+2\sigma^{-1} \lambda_{k-1}\right]
    \end{aligned}
  \end{equation}
  such that the \eqref{eq: solv cond} in $\Gamma(2 \delta) \backslash \Gamma$ holds and $g_{k-1}$ is smooth because of Lemma \ref{lem: extension interface}.

  \medskip
  
\noindent
  \textit{Ad matching-conditions.} The identity $\partial_x^m \partial_\rho^n \partial_t^l\left[c_k( \pm \rho, x, t)-c_k^{ \pm}(t)\right]= {O}\left(e^{-\alpha \rho}\right)$ as $\rho \rightarrow \infty$ in $\mathrm{L}^\infty(\Gamma(2\delta))$ holds since $c_k^\pm$ is the limit of $\rho\to\infty$ of the right-hand side of \eqref{eq: coeff inner eq} divided by $f''(\pm 1)$, thus by construction this yields the exponential decay.
\end{proof}

\subsection{Calculation of Mass-Conserving Term $\lambda^\varepsilon$}\label{subsec: lambda}
The goal is to identify the coefficients of the power expansion in the mass-conserving term $\lambda_\mathrm{A}$, as outlined in the beginning of this section. The calculation of the coefficients is based upon the mass conservation of the approximate solution, i.e., $\int_{\Omega} \partial_t \bar{c}_\mathrm{A}(x, t) \dif x = 0$. By \cite[Section 5.4]{ChenHilhorstLogak}, we can conclude:

\begin{thm}\label{lem: coeff lambda} The coefficients of the asymptotic expansion $\hat{\lambda}_\mathrm{A}(t)=\sum_{i=0}^{N+1}\varepsilon^i \lambda_i(t)$ for $t\in[0,T_0]$ are
  \begin{equation} \label{eq: lambda i}
    \begin{aligned}
      \sigma^{-1} \lambda_{0}  & =\overline{H_{\Gamma_{t}}}- \overline{\no_{\Gamma_t}\cdot \ve},                                                                                                            \\
      2\sigma^{-1}\lambda_i(t) & =-\overline{V_{\Gamma_{t}} H_{\Gamma_{t}} h_{i}} -\overline{\mathbf{v}_0\cdot \nabla^{\Gamma} h_{i}}+ \Lambda_{i-1},\quad i \geqslant 1,
    \end{aligned}
  \end{equation}
  where $\Lambda_{i-1}(t)$ depends only on expansions of order less equal than $i-1$
  and $$\overline{\phi(\cdot)}\coloneqq \frac{1}{|\mathbb{T}^1|} \int_{\mathbb{T}^1} \phi(s)|X_0'(s)| \dif s=\frac{1}{|\Gamma_t \mid} \int_{\Gamma_t} \phi(x, t) \dif \mathcal{H}^{1} .$$
\end{thm}
\begin{proof}
  Done as in \cite[Section 5.4]{ChenHilhorstLogak} with additional convection terms. We note that we use the opposite sign convention for $f$ and $\lambda_\eps$, $\lambda_i$.
\end{proof}

\begin{rem}
By entering the identity of the mass coefficients \eqref{eq: lambda i} in \eqref{eq: time der of h}, the $\lambda_i$ can be replaced. Additionally, \eqref{eq: time der of h} depends on some mean values of coefficients of height functions, other than the higher order terms in (A.71) from \cite{AbeFei23}.
\end{rem}
\iffalse The proof is based on the following key ingredients: firstly, the mass conservation condition over $\Omega$ must be used; secondly, the integral must be decomposed into two subdomains separated by the zero level set of $c_\varepsilon$. A suitable change of variables must then be employed, the inner and outer expansions must be plugged in, and the terms must be matched in order to determine the $\lambda_i$.
\fi

\noindent
\begin{proof*}{of Theorem \ref{thm:Approx1}} By construction, one can verify the claimed estimates of Theorem \ref{thm:Approx1} analogously as, e.g., in \cite[Section 4]{AbelsMarquardt2}. In our case, the estimates are even simpler since there is no $\varepsilon^{M-\frac{1}{2}}$-order term as in \cite{AbelsMarquardt2}.    
\end{proof*}

\def\ocirc#1{\ifmmode\setbox0=\hbox{$#1$}\dimen0=\ht0 \advance\dimen0
    by1pt\rlap{\hbox to\wd0{\hss\raise\dimen0
        \hbox{\hskip.2em$\scriptscriptstyle\circ$}\hss}}#1\else {\accent"17 #1}\fi}

\bigskip

\noindent
{\it
  (H. Abels) Fakult\"at f\"ur Mathematik,
  Universit\"at Regensburg,
  93040 Regensburg,
  Germany}\\
{\it E-mail ad\dif{r}ess: {\sf helmut.abels@mathematik.uni-regensburg.de} }\\[1ex]
{\it
(H. Mumtaz) Fakult\"at f\"ur Mathematik,
Universit\"at Regensburg,
93040 Regensburg,
Germany}\\
{\it E-mail ad\dif{r}ess: {\sf hanifah.mumtaz@mathematik.uni-regensburg.de} }\\[1ex]

\begin{thebibliography}{10}

  \bibitem{AbeFei23}
  H.~Abels and M.~Fei,
  \newblock Sharp interface limit for a {N}avier--{S}tokes/{A}llen--{C}ahn system with different viscosities,
  \newblock {\em SIAM J. Math. Anal.} \textbf{55} (2023), no.~4, 4039--4088.

  \bibitem{AbelsConvectiveAC}
  H.~Abels.
  \newblock ({N}on-)convergence of solutions of the convective {A}llen-{C}ahn equation.
  \newblock {\em Partial Differ. Equ. Appl.} 3(1), 2022.

  \bibitem{AbelsMarquardt1}
H.~Abels and A.~Marquardt.
\newblock Sharp interface limit of a {S}tokes/{C}ahn-{H}illiard system, {P}art
  {I}: {C}onvergence result.
\newblock {\em Interfaces Free Bound.} 23(3):353--402, 2021.

  \bibitem{AbelsMarquardt2}
  H.~Abels and A.~Marquardt.
  \newblock Sharp interface limit of a {S}tokes/{C}ahn-{H}illiard system, {P}art {II}: {A}pproximate solutions.
  \newblock {\em J. Math. Fluid Mech.} 23(2):Paper No. 38, 48, 2021.

  \bibitem{AbelsMoserNSAC}
  H.~Abels and M.~Moser.
  \newblock Well-posedness of a {N}avier-{S}tokes/mean curvature flow system.
  \newblock In {\em Mathematical analysis in fluid mechanics---selected recent results}, volume 710 of {\em Contemp. Math.}, pages 1--23. Amer. Math. Soc.,
  [Providence], RI, [2018] \copyright 2018.

  \bibitem{StokesAllenCahn}
  H.~Abels and Y.~Liu.
  \newblock Sharp interface limit for a {S}tokes/{A}llen-{C}ahn system.
  \newblock {\em Arch. Ration. Mech. Anal.} 229(1):417--502, 2018.

  \bibitem{AlikakosLimitCH}
  N.~D. Alikakos, P.~W. Bates, an\dif{x}.~Chen.
  \newblock Convergence of the {C}ahn-{H}illiard equation to the {H}ele-{S}haw model.
  \newblock {\em Arch. Rational Mech. Anal.} 128(2):165--205, 1994.

  \bibitem{Allen1977}
  J.~Cahn and S.~M.~Allen,
  \newblock A microscopic theory for domain wall motion and its experimental verification in Fe-Al alloy domain growth kinetics,
  \newblock {\em J. Phys. Colloques} 38:C7-51--C7-54, 1977.
  \newblock \doi{10.1051/jphyscol:1977709}.

  \bibitem{Allen19791085}
  S.~M.~Allen and J.~W.~Cahn,
  \newblock A microscopic theory for antiphase boundary motion and its application to antiphase domain coarsening,
  \newblock {\em Acta Metallurgica} 27(6):1085--1095, 1979.
  \newblock \doi{10.1016/0001-6160(79)90196-2}.

  \bibitem{Analyticity}
  Y.~Giga.
  \newblock Analyticity of the semigroup generated by the Stokes operator in {$L_{r}$} spaces.
  \newblock {\em Math. Z.} 178:297--329, 1981.

  \bibitem{BoyerModelH}
  F.~Boyer.
  \newblock Mathematical study of multi-phase flow under shear through order parameter formulation.
  \newblock {\em Asymptot. Anal.} 20(2):175--212, 1999.

  \bibitem{CahnHilliard}
  J.~W. Cahn and J.~E. Hilliard.
  \newblock Free energy of a nonuniform system. {I}. {I}nterfacial energy.
  \newblock {\em J. Chem. Phys.} 28, No. 2:258--267, 1958.

  \bibitem{ChenHilhorstLogak}
  X.~Chen, D.~Hilhorst, and E.~Logak.
  \newblock Mass conserving {A}llen-{C}ahn equation and volume preserving mean curvature flow.
  \newblock {\em Interfaces Free Bound.} 12(4):527--549, 2010.

  \bibitem{DeMottoniSchatzman}
  P.~De~Mottoni and M.~Schatzman.
  \newblock Geometrical evolution of devoloped interfaces.
  \newblock {\em Trans. Amer. Math. Soc.} 347(5):1533--1589, 1995.

  \bibitem{ESS}
  L.~C.~Evans, H.~M.~Soner, and P.~E.~Souganidis
  \newblock Phase transitions and generalized motion by mean curvature.
  \newblock {\em Comm.~Pure Appl.~Math.} 45:1097--1123, 1992.

  \bibitem{Fei}
  M.~Fei.
  \newblock Global sharp interface limit of the {H}ele-{S}haw-{C}ahn-{H}illiard system.
  \newblock {\em Math. Methods Appl. Sci.} 40(3):833--852, 2017.

  \bibitem{FischerLauxSimon}
  J.~Fischer, T.~Laux, T.~M.~Simon
  \newblock Convergence rates of the Allen-Cahn equation to mean curvature flow: A short proof based on relative entropies.
  \newblock {\em SIAM J. Math. Anal.} 52(6):6222--6233, 2020.

  \bibitem{GalGrasselliDCDS}
  C.~G. Gal and M.~Grasselli.
  \newblock Longtime behavior for a model of homogeneous incompressible two-phase flows.
  \newblock {\em Discrete Contin. Dyn. Syst.} 28(1):1--39, 2010.

  \bibitem{GiorginiGrasselliWu}
  A.~Giorgini, M.~Grasselli, and H.~Wu.
  \newblock Diffuse interface models for incompressible binary fluids and the mass-conserving {A}llen-{C}ahn approximation.
  \newblock {\em J. Funct. Anal.} 283(9):109631, 2022.

  \bibitem{GurtinTwoPhase}
  M.~E. Gurtin, D.~Polignone, and J.~Vi{\~n}als.
  \newblock Two-phase binary fluids and immiscible fluids described by an order parameter.
  \newblock {\em Math. Models Methods Appl. Sci.} 6(6):815--831, 1996.

  \bibitem{HenselLiuModelH}
  S.~Hensel, Y.~Liu.
  \newblock The sharp interface limit of a Navier--Stokes/Allen-Cahn system with constant mobility: Convergence rates by a relative energy approach.
  \newblock {\em SIAM J. Math. Anal.} 55(5):4751--4787, 2023.

  \bibitem{HohenbergHalperin}
  P.~Hohenberg and B.~Halperin.
  \newblock Theory of dynamic critical phenomena.
  \newblock {\em Rev. Mod. Phys.} 49:435--479, 1977.

  \bibitem{Ilmanen}
  T.~Ilmanen
  \newblock Convergence of the Allen-Cahn equation to Brakke's motion by mean curvature.
  \newblock {\em J.~Differential Geom.} 38:417--461, 1993.

  \bibitem{TwoPhaseVariableDensityJiangEtAl}
  J.~Jiang, Y.~Li, and C.~Liu.
  \newblock Two-phase incompressible flows with variable density: an energetic variational approach.
  \newblock {\em Discrete Contin. Dyn. Syst.} 37(6):3243--3284, 2017.

  \bibitem{PreprintRemarks}
  S.~Jiang, X.~Su, F.~Xie.
  \newblock Remarks on Sharp Interface Limit for an Incompressible Navier-Stokes and Allen-Cahn Coupled System.
  \newblock {\em Chin. Ann. Math. Ser. B.} 44:663--686, 2023.

  \bibitem{Kagaya}
  T.~Kagaya
  \newblock Convergence of the Allen-Cahn equation with a zero Neumann boundary condition on non-convex domains.
  \newblock {\em Math.~Ann.} 373:1485--1528, 2019.

  \bibitem{KKR}
  M.~Katsoulakis, G.~T.~Kossioris, F.~Reitich
  \newblock Generalized Motion by Mean Curvature with Neumann Conditions and the Allen-Cahn Model for Phase Transitions.
  \newblock {\em The Journal of Geometric Analysis} 5(2):255--279, 1995.

  \bibitem{KroemerLaux}
    M.~Kroemer and T.~Laux  
    \newblock Quantitative convergence of the nonlocal Allen-Cahn equation to volume-preserving mean curvature flow. 
    \newblock {\em Math. Ann.} 391(3):4455--4472, 2025.


\bibitem{LauxSimon}
  T.~Laux, T.~M.~Simon
  \newblock Convergence of the Allen-Cahn Equation to Multiphase Mean Curvature Flow.
  \newblock {\em Comm.~Pure Appl.~Math.} 71(8):1493--1714, 2018.

  \bibitem{LeeLowengrub1}
  H.-G. Lee, J.~S. Lowengrub, and J.~Goodman.
  \newblock Modeling pinchoff and reconnection in a {H}ele-{S}haw cell. {I}.
  {T}he models and their calibration.
  \newblock {\em Phys. Fluids} 14(2):492--513, 2002.

  \bibitem{LiuSatoTonegawa2}
  C.~Liu, N.~Sato, and Y.~Tonegawa.
  \newblock Two-phase flow problem coupled with mean curvature flow.
  \newblock {\em Interfaces Free Bound.} 14(2):185--203, 2012.

  \bibitem{LiuShenModelH}
  C.~Liu and J.~Shen.
  \newblock A phase field model for the mixture of two incompressible fluids and its approximation by a {F}ourier-spectral method.
  \newblock {\em Phys. D} 179(3-4):211--228, 2003.

  \bibitem{PhDMarquardt}
  A.~Marquardt.
  \newblock {\em Sharp Interface Limit for a Stokes / Cahn-Hilliard System}.
  \newblock PhD thesis, University Regensburg, urn:nbn:de:bvb:355-epub-384308, 2019.

  \bibitem{MizunoTonegawa}
  M.~Mizuno, Y.~Tonegawa
  \newblock Convergence of the Allen-Cahn equation with Neumann boundary conditions.
  \newblock {\em SIAM J.~Math.~Anal.} 47(3):1906--1932, 2015.

  \bibitem{PromotionStefan}
  S.~Schaubeck.
  \newblock {\em Sharp interface limits for diffuse interface models}.
  \newblock PhD thesis, University Regensburg, urn:nbn:de:bvb:355-epub-294622, 2014.

  \bibitem{SchumacherInstNSt}
  K.~Schumacher.
  \newblock The instationary {N}avier-{S}tokes equations in weighted {B}essel-potential spaces.
  \newblock {\em J. Math. Fluid Mech.}, 11(4):552--571, 2009.

  \bibitem{Stein:SingInt}
  E.~M. Stein.
  \newblock {\em Singular Integrals and Differentiability Properties of Functions}.
  \newblock Princeton Hall Press, Princeton, New Jersey, 1970.

  \bibitem{WangZhangHSCH}
  X.~Wang and Z.~Zhang.
  \newblock Well-posedness of the {H}ele-{S}haw-{C}ahn-{H}illiard system.
  \newblock {\em Ann. Inst. H. Poincar\'e Anal. Non Lin\'eaire}, 30(3):367--384, 2013.

\end{thebibliography}
\end{document}